\documentclass{article}
\usepackage{graphicx} 
\usepackage{amsmath}
\usepackage{amsthm}
\usepackage[margin=1.3in]{geometry}
\usepackage[shortlabels]{enumitem}
\usepackage{amssymb}
\usepackage{array}
\usepackage{multirow}
\usepackage{tikz-cd}

\newcommand{\R}{\mathbb{R}}
\newcommand{\C}{\mathbb{C}}
\newcommand{\g}{\mathfrak{g}}
\renewcommand{\a}{\mathfrak{a}}
\newcommand{\su}{\mathfrak{su}}
\newcommand{\p}{\mathfrak{p}}
\renewcommand{\b}{\mathfrak{b}}
\newcommand{\so}{\mathfrak{so}}
\renewcommand{\sp}{\mathfrak{sp}}
\renewcommand{\H}{\mathbb{H}}
\renewcommand{\sl}{\mathfrak{sl}}
\newcommand{\h}{\mathfrak{h}}

\renewcommand{\O}{\mathrm{O}}
\newcommand{\tr}{\mathrm{tr}}
\newcommand{\ad}{\mathrm{ad}}
\renewcommand{\c}{\mathfrak{c}}
\renewcommand{\k}{\mathfrak{k}}
\renewcommand{\p}{\mathfrak{p}}

\newcommand{\I}{\hat{I}}

\newcommand{\z}{\mathfrak{z}}
\newcommand{\s}{\mathfrak{s}}
\newcommand{\im}{\mathrm{im}}

\newtheorem{theorem}{Theorem}[section]
\newtheorem{proposition}[theorem]{Proposition}
\newtheorem{lemma}[theorem]{Lemma}
\newtheorem{corollary}[theorem]{Corollary}
\newtheorem{conjecture}[theorem]{Conjecture}
\newtheorem{example}[theorem]{Example}
\newtheorem{remark}[theorem]{Remark}
\newtheorem*{MainTheorem}{Theorem \ref{thm::MainTheorem}}

\title{FAILURE OF LICHNEROWICZ-TYPE RESULT IN PARABOLIC GEOMETRIES OF REAL RANK AT LEAST $3$}
\author{TOBY ALDAPE}
\date{September 2025}

\begin{document}

\maketitle
\begin{abstract}
   Given a Yamaguchi nonrigid parabolic model geometry $(G,P)$ with $G$ simple of real rank at least $3$, we use techniques developed by Erickson to establish the existence of closed, nonflat, essential, regular, normal Cartan geometries modeled on $(G,P)$. Yamaguchi nonrigidity is a necessary condition for admitting nonflat, regular, normal examples. This rules out Lichnerowicz-type conjectures for these model geometries.
\end{abstract}
\section{Introduction}
A conformal manifold $(M,c)$ is called \emph{essential} if, for any metric $g$ in the conformal class $c$, $\mathrm{Isom}(g)\subsetneq \mathrm{Conf}(c)$. That is, the full conformal automorphism group is larger than the isometry group of any representative metric. As conjectured by Lichnerowicz and proven independently by Lelong-Ferrand \cite{Lelong-Ferrand} and Obata \cite{Obata}, there are remarkably few essential conformal manifolds. 
\begin{theorem}[Ferrand-Obata]\label{Ferrand-Obata}
Let $(M,[g])$ be a connected, essential, Riemannian conformal manifold. Then $M$ is conformally diffeomorphic to either the round sphere or Euclidean space.
\end{theorem}
Riemannian conformal manifolds are the regular, normal, parabolic Cartan geometries modeled on the conformal sphere. This is the homogeneous space for the action of the orthogonal group $\mathrm{O}(n)$ by M{\"o}bius transformations. This model geometry has stabilizer a parabolic subgroup of $O(n).$ 

CR manifolds satisfy an analogous result to Ferrand-Obata. We will call such results \emph{Lichnerowicz-type}. A pseudoconvex CR manifold is a $2n+1$ dimensional manifold $M$ with a $2n$ dimensional co-oriented distribution $H\subset TM$ admitting an almost complex structure $J:H\to H$ for which $J^2=-1$ and satisfying certain positive definiteness and integrability conditions. The subbundle $H$ induces an equivalence class of contact forms $[\theta]$ vanishing on $H.$ These satisfy the property that $d\theta: H\times H\to \R$ is the imaginary part of a positive definite Hermitian form. Call a CR manifold $(M,H,J)$ essential if $\mathrm{Aut}(M,\theta, J)\subsetneq \mathrm{Aut}(M,H,J)$ for all contact forms $\theta$ in the equivalence class. Then results of Webster \cite{Webster} and Schoen \cite{Schoen} each imply the following.
\begin{theorem}[Schoen-Webster]\label{Webster-Schoen}
If a $2n+1$-dimensional compact CR manifold is essential then it is CR diffeomorphic to $S^{2n+1}$ with its standard CR structure.
\end{theorem}
To make the connection to regular, normal Cartan geometries, we have to generalize the CR condition slightly to partially integrable, almost CR. A pseudoconvex, partially integrable, almost CR manifold is equivalent to a choice of regular, normal Cartan geometry modeled on the CR sphere $S^{2n+1}$, a homogeneous space for the action of $\mathrm{SU}(n+1,1)$ by CR diffeomorphisms. In particular, the group $\mathrm{SU}(n+1,1)$ has real rank one and its action on the sphere has a parabolic stabilizer subgroup. Frances \cite{Frances1} generalized the Ferrand-Obata and Schoen-Webster theorems, proving a Lichnerowicz-type result that applies to all real rank one parabolic geometries. Alt \cite{Alt} modified Frances's result to prove the following, under an appropriate Cartan-geometric definition of essentiality.
\begin{theorem}\label{FrancesAlt}
Let $\mathcal{G}\to M$ be a regular, real rank one parabolic geometry. If the parabolic structure is essential, then $\mathcal{G}$ is geometrically isomorphic to either the compact homogeneous model $G/P$ or the noncompact $G/P\backslash \{eP\}$.
\end{theorem}
There are four series of real rank one regular, normal parabolic geometries, corresponding to \textbf{(1)} conformal, \textbf{(2)} pseudoconvex, partially integrable, almost CR, \textbf{(3)} quaternionic contact, and \textbf{(4)} octonionic contact structures. These geometric structures are modeled on the homogeneous spaces \textbf{(1)} $\partial \mathbf{H}^{n+1}_{\R},$ \textbf{(2)} $\partial \mathbf{H}^{n+1}_{\C}$,  \textbf{(3)} $\partial \mathbf{H}^{n+1}_{\mathbb{H}}$, and \textbf{(4)} $\partial \mathbf{H}^2_{\mathbb{O}}$, the boundaries of real, complex, quaternionic, and octonionic hyperbolic spaces.

 D'Ambra and Gromov \cite{DambraGromov} asked if it was also true in higher signature that an essential, closed, pseudo-Riemannian conformal manifold must be flat. Frances \cite{Frances2} answered this question negatively, proving the existence of infinitely many closed, nonflat, essential conformal manifolds in each signature $(p,q)$ with $2\le p\le q$. Furthermore, Case, Curry and Matveev \cite{CaseCurryMatveev} proved that there are essential, closed, nondegenerate CR manifolds of signature $(p,q)$ with $2\le p\le q$. Both of these situations correspond to regular, normal parabolic geometries modeled on homogeneous spaces for groups of real rank at least $3$, so these cases are not addressed by Frances and Alt's Theorem \ref{FrancesAlt}.
 
 With a more general refutation of Lichnerowicz-type results for parabolic geometries in mind, Erickson \cite{Erickson2} built a Cartan geometry associated to a fixed harmonic curvature form called a curvature tree. This construction globalizes a local construction of Kruglikov and The \cite{KruglikovThe}. Taking compact quotients of the curvature tree, Erickson developed a procedure for exhibiting closed, nonflat parabolic geometries admitting essential transformations for parabolic model geometries $(G,P)$ such that $G$ is simple of real rank at least $3.$ This paper applies that procedure in all sensible cases.
 
 There are many parabolic geometries in which, by the vanishing of a certain cohomology module, there are no nonflat, regular, normal examples. These parabolic model geometries are called Yamaguchi rigid. For Yamaguchi rigid model geometries, it is vacuously true that all nonflat Cartan geometries are not essential. On the other hand, Yamaguchi compiled a list (in \cite{Yamaguchi1} with minor corrections in \cite{Yamaguchi2}) of all infinitesimal parabolic model geometries $(\g,\p)$ that are not Yamaguchi rigid. In this paper we perform Erickson's procedure for each of the model geometries on Yamaguchi's list having real rank at least $3$, proving the existence of closed, nonflat, essential manifolds for all such parabolic model geometries. Thus, all such geometries fail a Lichnerowicz-type conjecture.
\begin{theorem}[Main Theorem]\label{thm::MainTheorem}
    Suppose $(G,P)$ is a Yamaguchi nonrigid parabolic model geometry with $G$ real simple of real rank at least $3$. Then there exists a closed, nonflat, locally homogeneous, regular, normal, parabolic Cartan geometry modeled on $(G,P)$ and admitting essential transformations. 
\end{theorem}
We can put this result in the context of the so far unresolved Lorentzian Lichnerowicz conjecture, on which there has been significant progress \cite{Mehidi}\cite{MelnickFrances}\cite{MelnickPecastaing}\cite{Pecastaing1}\cite{Pecastaing2}.
\begin{conjecture}[Lorentzian Lichnerowicz]
If $M$ is an essential, closed, Lorentzian conformal manifold, then $M$ is conformally flat.
\end{conjecture}
Lorentzian conformal manifolds of dimension $n\ge 3$ are regular, normal, parabolic Cartan geometries modeled on a homogeneous space for the group $\O(n,2).$ Since $\O(n,2)$ for $n\ge 3$ is simple of real rank $2,$ Lorentzian Lichnerowicz falls into an intermediate case between Frances and Alt's Theorem \ref{FrancesAlt} for real rank $1$ - a domain where Lichnerowicz-type results hold - and our Theorem \ref{thm::MainTheorem} for real rank at least $3$ - a domain where they fail.

While this paper does not address failure of Lichnerowicz-type conjectures in real rank $2$, we believe that the Erickson techniques can be used to construct noncompact, nonflat, essential geometries for many but not all of of the Yamaguchi nonrigid parabolic model geometries modeled on homogeneous spaces for simple groups of real rank $2$. However, the construction of an essential compact quotient seems to rely in an important way on the assumption of real rank at least $3$.

\section*{Acknowledgments}
I am particularly grateful to Jacob Erickson for suggesting the problem, setting up the framework needed to address it, and for his thoughtful feedback on multiple drafts of this paper.
\section{Lie Theory Background}
\subsection{Conventions}
Given a real vector space $V$, we denote its complexification by $V^{\C}$. If $V$ is a complex vector space, $\overline{V}$ is the same vector space but with the conjugate complex structure. We call a real Lie algebra \emph{noncomplex} it is not the underlying real Lie algebra of some complex Lie algebra. If $\g$ is a Lie algebra and $\h$ is a subalgebra, then $\z(\g)$ represents the center of $\g$ while $\z_{\g}(\h)$ represents the centralizer of $\h$ in $\g$. Given an inner product, a vector $v$ determines a covector $v_{\flat}$ and a covector $\alpha$ determines a vector $\alpha^{\sharp}$ via the musical isomorphisms. If $\alpha$ is a root or restricted root of a Lie algebra $\g$, $\eta_{\alpha}$ refers to some nonzero vector in the rootspace or restricted rootspace $\g_{\alpha}.$ The notation $\alpha_k$ refers to a simple root, while $\beta_k$ refers to a simple restricted root. For simple roots $\alpha_i,\alpha_j,$ the notation $(ij)$ refers to the composition of simple root reflections $s_i\circ s_j.$ The letter $\mu$ always refers to the highest root of a Lie algebra. Coefficients of the Cartan matrix are written $c_{ij}.$
\subsection{Structure Theory of Real Lie Algebras}
 Given a real semisimple Lie algebra $\g$, we assume a fixed choice of Cartan involution $\theta:\g\to \g$, a maximally noncompact $\theta$-stable Cartan subalgebra $\c\le \g$ with noncompact part $\a$, complexification $\h:=\c^{\C}\le \g^{\C}$, and a root system $\Delta\subset \h^*$ for $\g^{\C}.$ Let $\sigma:\g^{\C}\to \g^{\C}$ be the conjugation about $\g$ and define $\sigma^*:\Delta\to \Delta$ by $\alpha\mapsto \overline{\alpha\circ \sigma}.$ Define the set of \emph{compact roots} 
\[
\Delta_c=\{\alpha\in \Delta: \sigma^*\alpha=-\alpha\}.
\]
We may fix a positive subsystem $\Delta^+\subset \Delta$ such that $\Delta^+\backslash \Delta_c$ is preserved by $\sigma^*.$ For $\alpha\in \Delta,$ we have $\alpha|_{\a}=0$ exactly when $\alpha\in \Delta_c.$ Then define \emph{restricted roots}
\[
\hat{\Delta}=\{\alpha|_{\a}:\alpha\in \Delta\backslash \Delta_c\}\subset \a^*.
\]
We get the restricted root space decomposition
\[
\g=\z_{\g}(\a)\oplus 
\bigoplus_{\alpha\in \hat{\Delta}}\g_{\alpha}
 \]
 where 
 \[
 \g_{\alpha}=\{X\in \g: \ad(H)(X)=\alpha(H)X \ \mathrm{for} \ H\in \a\}
 \]
is a \emph{restricted root space}. There is an induced positive subsystem $\hat{\Delta}^+\subset \hat{\Delta}$ obtained by restricting all roots of $\Delta^+\backslash \Delta_c$ to $\a.$

Given a simple system $\Delta^0$ for $\Delta^+$, the restrictions to $\a$ of the roots in $\Delta^0\backslash \Delta_c$ form a simple system $\hat{\Delta}^0$ for $\hat{\Delta}^+,$ with some pairs of simple roots in $\Delta^0$ restricting to a single simple restricted root in $\hat{\Delta}^0.$ Compact simple roots are excluded because they all restrict to $0.$ A subset of simple restricted roots $\I \subset \hat{\Delta}^0$ determines a parabolic subalgebra by the following process. Given a restricted root 
\[
\alpha=\sum_{\beta\in \Delta^0}n_{\beta}\cdot\beta\in \hat{\Delta},
\]
define its $\hat{I}$-\emph{height} by
\[
h_{\hat{I}}(\alpha):=\sum_{\beta\in \hat{I}}n_{\beta}.
\]
Then $\p\le \g$ is defined by
\[
\p=\z_{\g}(\a)\oplus 
\bigoplus_{\alpha\in \hat{\Delta}:h_{\hat{I}}(\alpha)\ge 0} \g_{\alpha}.
\]
Notably, $\p=\g$ if and only if $\hat{I}=\varnothing$. Up to inner automorphism, this gives every parabolic subalgebra of $\g.$ Therefore we will assume without loss of generality that a given parabolic subalgebra is in this form. Define 
\[
\hat{\Delta}(\g_0):=\{\alpha\in \hat{\Delta}: h_{\I}(\alpha)=0\}
\]
and
\[
\hat{\Delta}^+(\p^+):=\{\alpha\in \hat{\Delta}: h_{\hat{I}}(\alpha)>0\}.
\]
For $\hat{I}\neq \varnothing$, the restriction $\mu|_{\a}$ of the highest root of $\g^{\C}$ is always in $\hat{\Delta}^+(\p^+).$
The nilradical of $\p$ is 
\[
\p_+:=\bigoplus_{\alpha\in \hat{\Delta}^+(\p^+)} \g_{\alpha}.
\]
The \emph{Levi} subalgebra is the reductive subalgebra 
\[
\g_0:=\z_{\g}(\a)\oplus \bigoplus_{\alpha\in \hat{\Delta}(\g_0)} \g_{\alpha},
\]
which is isomorphic to the quotient $\p/\p_+$ of $\p$ by its nilradical. Since $\g_0$ is reductive, its derived subalgebra $\g_0^{ss}:=[\g_0,\g_0]$ is semisimple. One possible source of confusion with the notation $\g_0$ is that it can be mistaken for the restricted root space $\g_{\alpha}$ when $\alpha=0.$ However, $0$ is not considered a restricted root and this space is instead written $\z_{\g}(\a).$
\begin{proposition}\label{prop::ComplementarySpaces}
The complementary subspaces $\z(\g_0),\g_0^{ss}\le \g$ are Killing-orthogonal.
\end{proposition}
\begin{proof}
    For $X\in \z(\g_0)$ and $Y,Z\in \g_0^{ss},$
    \[
    \langle X,[Y,Z]\rangle=\langle [X,Y],Z\rangle=0.
    \]
    Since $\g_0^{ss}$ is perfect, the claim follows.
\end{proof}
Observe that if $\beta_k\in \hat{\Delta}^0\backslash \hat{I}$ then $\g_{\beta_k}\le \g_0.$ Then, for $X\in \z(\g_0)\cap \a,$ we must have 
\[
0=\beta_k(X)=\langle X,\beta_k^{\sharp} \rangle.
\]
It follows from Proposition \ref{prop::ComplementarySpaces} that $\beta_k^{\sharp}\in \g_0^{ss}\cap \a.$ From their definitions, we know that $\p=\g_0\oplus \p_+.$ In fact, $h_{\hat{I}}$ induces a grading 
\[
\g=\g_{-k}\oplus \cdots \oplus \g_0 \oplus \cdots \oplus\g_k,
\] 
and we define $\g_-:=\g_{-k}\oplus \cdots \oplus\g_{-1}$ and we already have $\p_+=\g_1\oplus \cdots \oplus\g_k,$ so that $\g=\g_-\oplus \g_0 \oplus \p_+.$ There exists a special element $E\in \a$ such that $[E,X]=iX$ for $X\in \g_i,$ known as the \emph{grading element}. In particular, this definition implies $E\in \z(\g_0).$

For every $\alpha\in \Delta^0\backslash \Delta_c$, there exists a unique $\overline{\alpha}\in \Delta^0\backslash \Delta_c$ such that
\[
\sigma^*(\alpha)=\overline{\alpha}+\sum_{\beta\in \Delta_c\cap \Delta^0} n_{\beta}\cdot \beta
\]
for some integers $n_{\beta}$. Then $\alpha|_{\a}=\overline{\alpha}|_{\a}.$ To draw the \emph{Satake diagram} for a Lie algebra $\g,$ take the Dynkin diagram for $\g^{\C},$ color the elements of $\Delta_c\cap\Delta^0$ black, and for noncompact simple roots $\alpha,$ if $\alpha\neq \overline{\alpha},$ connect $\alpha$ and $\overline{\alpha}$ by a bi-directional arrow. The resulting diagram is independent of a choice of maximally noncompact Cartan subalgebra $\c$ and positive system $\Delta^+\subset \Delta$ compatible with $\sigma$. A full list of Satake diagrams of real Lie algebras $\g$ such that $\g^{\C}$ is simple is given in Appendix A. 

To pass from a Satake diagram to a set of simple restricted roots, delete any compact roots and glue together any pair of roots connected by a bi-directional arrow. The problem of determining inner products of simple restricted roots relative to the Killing form on $\a^*$ to obtain the restricted Dynkin diagram is a bit more subtle, but a table of the results is collected the Appendix A, and in particular connected subsets of simple roots correspond to connected subsets of simple restricted roots.

As discussed above, parabolic subalgebras correspond to subsets $\hat{I}\subset \hat{\Delta}^0.$ Considering simple roots of $\Delta^0$ that restrict to roots in $\hat{I},$ we get a subset $I\subset \Delta^0$ such that $I\cap \Delta_c=\varnothing$ and if $\alpha\in I$ then $\overline{\alpha}\in I.$ We will call a subset $I\subset \Delta^0$ compatible with $\g$ when these two conditions hold.

Similar to what we did above for the real semisimple case, $I\subset \Delta^0$ defines a height function $h_I$ on $\Delta\cup \{0\}$, and thereby determines a complex parabolic subalgebra $\p_I\le \g^{\C}.$ When $I$ is compatible with $\g$, there is a corresponding subset $\hat{I}\subset \hat{\Delta}^0$ given by the set of restrictions of elements of $I,$ and $\hat{I}$ induces a parabolic subalgebra $\p_{\hat{I}}\le \g$ for which $(\p_{\hat{I}})^{\C}=\p_I.$ Thus parabolics $\p_I\le\g_{\C}$ determined by subsets $I\subset \Delta^0$ compatible with $\g$ correspond one-to-one with the parabolics $\p_{\hat{I}}\le\g$ by complexification.

\subsection{Real and Complex Representations}
 If $\g$ is complex semisimple and $X_1,\dots X_n$ is a $\C$-basis for $\g,$ then $X_1,\dots X_n, iX_1,\dots iX_n$ is an  $\R$-basis for $\g.$ Then given a $\C$-linear endomorphism $T:\g\to \g,$ we have $\mathrm{tr}_{\R}(T)=2\mathrm{Re}(\tr_{\C}(T)).$ If $\g$ is a complex Lie algebra, let $\langle\cdot,\cdot \rangle_{\C}$ be the Killing bracket on $\g$, and let $\langle\cdot,\cdot \rangle_{\R}$ be the Killing bracket on $\g$ as a real Lie algebra. Then $\langle x, y \rangle_{\R}=2\mathrm{Re}(\langle x, y \rangle_{\C}).$ Let $\h_0\le\h$ be the subspace on which roots of $\g$ are real-valued. Because $\langle\cdot, \cdot \rangle_{\C}$ is real-valued on $\h_0$, if $x,y\in \h_0,$ then $\langle x,y \rangle_{\R}=2\langle x, y \rangle_{\C}.$ These brackets induce isomorphisms $\flat_{\C}:\h_0\mapsto \h_0^*$ and $\flat_{\R}:\h_0\to \h_0^*,$ which both induce Killing brackets on $\h_0^*.$ We have $\flat_{\R}=2\flat_{\C}.$ Calling their inverses $\sharp_{\R}$ and $\sharp_{\C},$ $\sharp_{\C}=2\sharp_{\R}.$

\begin{proposition}\label{prop::KillingBracketComparison}
Let $\langle \cdot, \cdot \rangle_{\C}, \langle \cdot, \cdot \rangle_{\R}$ be the complex and real Killing brackets induced on $\h_0^*$. Then
$\langle x, y \rangle_{\C}=2\langle x, y \rangle_{\R}.$
\end{proposition}
\begin{proof} For $x,y\in \h_0^*,$
    \[
    \langle x,y \rangle_{\C}
    =\langle x^{\sharp_{\C}}, y^{\sharp_{\C}} \rangle_{\C}
    =\frac{1}{2}\langle 2x^{\sharp_{\R}},2y^{\sharp_{\R}} \rangle_{\R}
    =2 \langle x, y \rangle_{\R},
    \]
    proving the claim.
\end{proof}

Given $\g$ real semisimple, a representation $\g\curvearrowright V$, and $\alpha\in \a^*,$ a vector $v\in V$ is called a \emph{weight vector} of weight $\alpha$ if $H\cdot v=\alpha(H)v$ for all $H\in \a.$ The \emph{weight space} $V_{\alpha}$ is the set of such vectors. A weight vector $v$ is called a \emph{lowest weight vector} if it is annihilated by all negative restricted rootspaces of $\g.$

\begin{example}
  Suppose $\g$ is the underlying real Lie algebra of a complex semisimple Lie algebra. Let $\h\le \g$ be the Cartan subalgebra. As explained in \cite{CapSlovak}, $\h\le \g$ is maximally noncompact with noncompact part $\h_0\le\h$, the subspace on which all roots of $\g$ are real valued. Then the restricted roots of $\g$ are exactly the roots of $\g$ restricted to $\h_0,$ and each restricted root space $\g_{\alpha}$ for $\alpha\in \hat{\Delta}$ is a $2$ dimensional real space. Furthermore, if $\g\curvearrowright V$ is a complex representation, the restricted weights of $\g\curvearrowright V$ are the weights of $\g\curvearrowright V$ restricted to $\h_0.$  
\end{example}

Given $\g$ real semisimple, define the \emph{fundamental weights} $\lambda_i\in \h^*$ of $\g^{\C}$ by the property 
$\langle \lambda_i, \alpha_j \rangle=\delta_{ij}\frac{|\alpha_j|^2}{2}$ for all simple roots $\alpha_j\in \Delta^0.$ An integral weight $\gamma\in \h^*$ decomposes as $\gamma=\sum \gamma^i \lambda_i$ for integers $\gamma^i.$ Define \emph{restricted fundamental weights} $\hat{\lambda}_i\in \a^*$ by the property $\langle \hat{\lambda}_i, \beta_j \rangle=\delta_{ij}\frac{|\beta_j|^2}{2}$ for all simple restricted roots $\beta_j\in \hat{\Delta}^0$. A restricted weight $\sigma$ decomposes as $\sigma=\sum \sigma^i \hat{\lambda}_i$ for integers $\sigma^i.$

\section{Cartan Geometry Background}
\subsection{Harmonic Curvature}\label{section::HarmonicCurvature}
The Killing form induces an isomorphism $\g_-^*\cong \p_+$. Let $\g\curvearrowright V$ be an action. Define $C_k(\p_+,V):=\bigwedge^k \p_+ \otimes V$ and $C^k(\g_-,V):=\bigwedge^k \g_-^* \otimes V.$ There is an isomorphism
\[
C_k(\p_+,V) \cong C^k(\g_-,V)
\]
of $\g_0$ modules. Let $\partial:C_*(\p_+,V)\to C_*(\p_+,V)$ be the boundary map for Lie algebra homology, and let $\partial^*:C^*(\p_+,V)\to C^*(\p_+,V)$ be the coboundary map for Lie algebra cohomology. Both are $\g_0$-equivariant. Then we may define the \emph{algebraic Laplacian operator} $\Delta:C^*(\g_-,V)\to C^*(\g_-,V)$ by $\Delta:=\partial \partial^*+\partial^* \partial,$ using the identifications between chains and cochains where appropriate. Elements of $\ker \Delta$ are called \emph{harmonic}. The operators $\partial$ and $\partial^*$ are adjoint relative to a certain positive definite inner product, and the algebraic Hodge theory of these spaces implies $\ker \Delta=\ker \partial \cap \ker \partial^*$, and the existence of canonical $\g_0$-equivariant isomorphisms
\[
H_*(\p_+,V)\cong \ker \Delta \cong H^*(\g_-,V).
\]

Now we specialize to the situation where $V=\g$ under the adjoint action. Recall that we have fixed a grading on $\g$ determined by $\p.$ As described in $\cite{CapSlovak},$ a parabolic Cartan geometry $(\mathcal{G},\omega)$ is called \emph{regular} if, at each point $p\in \mathcal{G},$ the curvature form
\[
\Omega_p\in \bigwedge^2 (\g/\p)^*\otimes \g \cong C_2(\p_+,\g)\cong C^2(\g_-,\g)
\]
is contained in positive homogeneity, so $\Omega_p\in C_2(\p_+,\g)_+ \cong C^2(\g_-,\g)_+.$ A parabolic Cartan geometry is called \emph{normal} if $\Omega_p\in \ker \partial^*$ for every $p\in \mathcal{G}.$ Given a Klein pair $(G,P),$ it is often natural to consider only the regular, normal Cartan geometries modeled on $(G,P).$ If $\Omega$ is normal, it determines at every point an equivalence class $\overline{\Omega}_p\in H_2(\p_+,\g)\cong H^2(\g_-,\g).$ Regularity in combination with normality implies $\overline{\Omega}_p\in H^2(\g_-,\g)_+$. The function $\overline{\Omega}$ is called the \emph{harmonic curvature} of the Cartan geometry. Harmonic curvature is a complete invariant obstructing flatness of a Cartan geometry. The following is a consequence of Theorem 3.1.12 of \cite{CapSlovak}.
\begin{theorem}
    For a regular, normal parabolic geometry, $\overline{\Omega}=0$ implies flatness of the Cartan geometry. 
\end{theorem}
If $H^2(\g_-,\g)_+=0$, the harmonic curvature $\overline{\Omega}$ always vanishes and so all regular, normal examples are flat. In this case, the pair $(G,P)$ is called \emph{Yamaguchi rigid}. It follows that the regular, normal parabolic geometries of interest are the Yamaguchi nonrigid geometries, classified by Yamaguchi in \cite{Yamaguchi1}, with some small mistakes corrected in \cite{Yamaguchi2}. This list is given in Appendix B.

\subsection{Kostant's Borel-Weil-Bott}
Let $\g$ be complex semisimple with a fixed parabolic subalgebra. This induces a grading on $\g$ and an associated decomposition $\g=\g_-\oplus \g_0 \oplus \p_+$ into the negative, zero, and positive graded components. For $\alpha_i\in \Delta^0,$ let $s_i:\h^*\to \h^*$ be the reflection about the hyperplane orthogonal to $\alpha_i$. Let $W^{\p}\subset W$ be the Hasse diagram associated to $\p$, as explained in \cite{CapSlovak}, and let $W^{\p}(i)\subset W^{\p}$ be the subset of elements of length $i.$ Given $\alpha_i\in \Delta^0,$ define $N(i)\subset \Delta^0$ to be the subset connected to $\alpha_i$ by an edge in the Dynkin diagram. The following is Recipe 3 of \cite{KruglikovThe}.

\begin{proposition}\label{prop::HasseDescription}
The Hasse diagram elements of length $2$ are characterized by 
\[
W^{\p}(2)=\{(ij): \alpha_i\in I, \alpha_j\in N(i)\cup I\}.
\]
\end{proposition}

Let $\lambda_1,\cdots,\lambda_n$ be the fundamental weights of $\g.$ A weight $\gamma=\sum \gamma^i \lambda_i$ is called $\g_0$-dominant if $\gamma^i\ge 0$ for every $i\in \Delta^0\backslash I.$ Equivalently, $\gamma|_{\h \cap \g_0^{ss}}$ is $\g_0^{ss}$-dominant. If $\gamma$ is a $\g_0$-dominant integral weight, let $V^{\gamma}$ be the irreducible representation of lowest weight $-\gamma.$ Let $\rho:=\sum \lambda_i$ be the Weyl vector. For $w\in W$ and $\lambda\in \h^*,$ define $w\cdot \lambda:=w(\lambda+\rho)-\rho.$ For $w\in W^{\p}$ and a $\g$-dominant weight $\lambda$, $w\cdot \lambda$ is always $\g_0$-dominant. Let $\Phi_w=w(\Delta^-)\cap \Delta^+.$ In particular, $\Phi_{(ij)}=\{\alpha_i,s_i(\alpha_j)\}.$ The following is a consequence of Theorem 3.3.5 in \cite{CapSlovak}.

\begin{theorem}[Kostant's Borel-Weil-Bott]
We have
\begin{enumerate}[(a)]
\item $H_{\C}^i(\g_-,V^{\gamma})\cong \bigoplus_{w\in W^{\p}(i)} V^{w\cdot \gamma}$
as $\g_0$-modules.
\item A harmonic representative of lowest weight $-w\cdot \gamma$ in $H_{\C}^i(\g_-,V^{\gamma})$ is given by 
\[
\bigwedge_{\alpha\in \Phi_w} (\eta_{\alpha})_{\flat} \otimes v_{-w(\gamma)}
\]
for $v_{-w(\gamma)}\in V$ a vector of weight $-w(\gamma).$
\end{enumerate}
\end{theorem}
Now suppose $\g$ is simple with highest root $\mu$. Then the adjoint representation $\g\curvearrowright \g$ has lowest weight $-\mu.$ It follows from part (b) of Kostant's Borel-Weil-Bott that harmonic representatives of the lowest weight vectors of $H_{\C}^2(\g_-,\g)$ are of the form 
\[
(\eta_{\alpha_i})_{\flat}
\wedge (\eta_{s_i(\alpha_j)})_{\flat}
\otimes \eta_{-w(\mu)},
\]
of $\g_0$ weight $-w\cdot \mu.$ 
We define 
\[
W^{\p}_+(i):=\{w\in W^{\p}(i): -w\cdot \mu(E)>0\}.
\]
These are the Weyl group elements corresponding to $\g_0$ irreducible subrepresentations in $H_{\C}^2(\g_-,\g)$ of positive homogeneity. Therefore
\[
H_{\C}^2(\g_-,\g)_+\cong \bigoplus_{w\in W^{\p}_+(2)} V^{w\cdot \mu}.
\]

A \emph{morphism of representations} $(\g,V)\to (\g',V')$ is a Lie algebra homomorphism $\phi:\g\to \g'$ and a linear map $\psi:V\to V'$ such that $\psi(X\cdot v)=\phi(X)\cdot \psi(v).$ A morphism of representations induces maps $C^*(\g,V)\to C^*(\g',V')$ and $C_*(\g,V)\to C_*(\g',V')$ which are $\partial$ and $\partial^*$-equivariant, respectively. In particular this holds if $\g'=\g$ with $\phi=\mathrm{Id}_{\g}$ and $\psi:V\to V'$ is $\g$-equivariant. The following is a slight extension of parts (1) and (2) of Proposition 3.3.6 in \cite{CapSlovak}.
\begin{proposition}\label{prop::ComplexVsRealCohomology}
Let $\g$ be a real semisimple Lie algebra. 
\begin{enumerate}[(a)]
    \item If $\g$ acts on a complex vector space $V$ by complex endomorphisms,
    \[
   H^*_{\R}(\g,V)\cong  H^*_{\C}(\g^{\C},V).
    \]
    Furthermore, map is the restriction of the isomorphism $C^*_{\R}(\g,V) \cong C^*_{\C}(\g^{\C},V)$ to harmonic elements.
     \item If $\g$ acts on a real vector space $V$ by real endomorphisms then
    \[
    H_{\R}^*(\g,V)^{\C}\cong H^*_{\C}(\g^{\C},V^{\C}).
    \]
    Furthermore, this isomorphism is the restriction of the isomorphism $C^*_{\R}(\g,V)^{\C}\cong C^*_{\C}(\g^{\C},V^{\C})$ to harmonic elements.
\end{enumerate}
\end{proposition}
\begin{proof}
\begin{enumerate}[(a)]
\item In the following diagram, the top isomorphism replaces real tensor products with complex ones, while the bottom isomorphism passes to the unique $\C$-linear extension. The left and right isomorphism are induced by the real and complex Killing forms on $\g$ and $\g^{\C}$ respectively. The reader may check that the diagram is commutative. 
\[\begin{tikzcd}
	{C_*^{\mathbb{R}}(\mathfrak{g},V)} & {C_*^{\mathbb{C}}(\mathfrak{g}^{\mathbb{C}},V)} \\
	{C^*_{\mathbb{R}}(\mathfrak{g},V)} & {C^*_{\mathbb{C}}(\mathfrak{g}^{\mathbb{C}},V)}
	\arrow["\cong",  from=1-1, to=1-2]
	\arrow["\cong"', from=1-1, to=2-1]
	\arrow["\cong",  from=1-2, to=2-2]
	\arrow["\cong"', from=2-1, to=2-2]
\end{tikzcd}\]
The top isomorphism on chains is $\partial^*$-equivariant, while the bottom isomorphism on cochains is $\partial$-equivariant. Identifying chains and cochains, it follows from the definition $\Delta=\partial \partial^*+\partial^* \partial$ that the two isomorphisms are $\Delta$-equivariant. It follows that harmonic cochains in $C_{\R}^*(\g,V)$ are mapped bijectively onto harmonic cochains in $C_{\C}^*(\g^{\C},V)$.
\item 
The following diagram commutes. The top and bottom arrows are $\partial^*$ and $\partial$-equivariant because $V\to V^{\C}$ is $\g$-equivariant.
\[\begin{tikzcd}
	{C_*^{\mathbb{R}}(\mathfrak{g},V)} & {C_*^{\mathbb{R}}(\mathfrak{g},V^{\mathbb{C}})} \\
	{C^*_{\mathbb{R}}(\mathfrak{g},V)} & {C^*_{\mathbb{R}}(\mathfrak{g},V^{\mathbb{C}})}
	\arrow[from=1-1, to=1-2]
	\arrow["\cong"', from=1-1, to=2-1]
	\arrow["\cong", from=1-2, to=2-2]
	\arrow[from=2-1, to=2-2]
\end{tikzcd}\]

Therefore the following diagram is also commutative with top and bottom arrows $\partial^*$ and $\partial$-equivariant.
\[\begin{tikzcd}
	{C_*^{\mathbb{R}}(\mathfrak{g},V)^{\mathbb{C}}} & {C_*^{\mathbb{R}}(\mathfrak{g},V^{\mathbb{C}})} \\
	{C^*_{\mathbb{R}}(\mathfrak{g},V)^{\mathbb{C}}} & {C^*_{\mathbb{R}}(\mathfrak{g},V^{\mathbb{C}})}
	\arrow["\cong", from=1-1, to=1-2]
	\arrow["\cong"', from=1-1, to=2-1]
	\arrow["\cong", from=1-2, to=2-2]
	\arrow["\cong"', from=2-1, to=2-2]
\end{tikzcd}\]
  It follows that harmonic cochains in $C^*_{\R}(\g,V)^{\C}$ are taken to harmonic cochains in $C^*_{\R}(\g,V^{\C}).$ Composing with the canonical isomorphism to $C^*_{\C}(\g^{\C},V^{\C})$ takes harmonic cochains bijectively onto harmonic cochains by part (a).
\end{enumerate}
\end{proof}
\begin{proposition}\label{prop::RealCohomologyOfComplex}
The forgetful map $C^*_{\C}(\g_-,V)\to C^*_{\R}(\g_-,\g)$ takes harmonic elements to harmonic elements. Consequently, there is a $\g_0$-equivariant injection $H^*_{\C}(\g_-,V)\to H^*_{\R}(\g_-,V).$
\end{proposition}
\begin{proof}
The following diagram is commutative.
\[\begin{tikzcd}
	{\mathfrak{p}_+} & {\mathfrak{p}_+\oplus \overline{\mathfrak{p}_+}} & {\mathfrak{p}_+^{\mathbb{C}}} \\
	{\mathfrak{g}_-^{*}} & {(\mathfrak{g}_-\oplus \overline{\mathfrak{g}_-})^{*}} & {(\mathfrak{p}^{\mathbb{C}})^*} \\
	\arrow[hook, from=1-1, to=1-2]
	\arrow["\cong"', from=1-1, to=2-1]
	\arrow["\cong", from=1-2, to=1-3]
	\arrow["\cong"', from=1-2, to=2-2]
	\arrow["\cong"', from=1-3, to=2-3]
	\arrow[hook, from=2-1, to=2-2]
	\arrow["\cong"', from=2-2, to=2-3]
\end{tikzcd}\]
As a result, the first two squares of the following diagram are commutative. We checked the commutativity of the third square in the proof of Proposition \ref{prop::ComplexVsRealCohomology} (a).
\[\begin{tikzcd}
	{C^{\mathbb{C}}_*(\mathfrak{p}_+,V)} & {C^{\mathbb{C}}_*(\mathfrak{p}_+\oplus \overline{\mathfrak{p}_+},V)} & {C^{\mathbb{C}}_*(\mathfrak{p}^{\mathbb{C}}_+,V)} & {C^{\mathbb{R}}_{*}(\mathfrak{p}_+,V)} \\
	{C^*_{\mathbb{C}}(\mathfrak{g}_-,V)} & {C^*_{\mathbb{C}}(\mathfrak{g}_-\oplus \overline{\mathfrak{g}_-},V)} & {C^*_{\mathbb{C}}(\mathfrak{g}_-^{\mathbb{C}},V)} & {C^*_{\mathbb{R}}(\mathfrak{g}_-,V)}
	\arrow[hook, from=1-1, to=1-2]
	\arrow["\cong"', from=1-1, to=2-1]
	\arrow["\cong", from=1-2, to=1-3]
	\arrow["\cong"', from=1-2, to=2-2]
	\arrow["\cong", from=1-3, to=1-4]
	\arrow["\cong"', from=1-3, to=2-3]
	\arrow["\cong"', from=1-4, to=2-4]
	\arrow[hook, from=2-1, to=2-2]
	\arrow["\cong"', from=2-2, to=2-3]
	\arrow["\cong"', from=2-3, to=2-4]
\end{tikzcd}\]

The composite of the bottom arrows is the forgetful map. Let $\overline{\p}_+$ and $\overline{\g_-}$ act trivially on $V.$ The first two top arrows are $\partial^*$-equivariant because $(\p_+,V)\to (\p_+\oplus \overline{\p_+}, V)$ and $(\p_+\oplus \overline{\p_+},V)\to (\p_+^{\C},V)$ are morphisms of representations. We checked that the third map is $\partial^*$-equivariant in the proof of Proposition \ref{prop::ComplexVsRealCohomology} (a) The bottom arrows are $\partial$-equivariant for similar reasons. The claim follows.
\end{proof}

\subsection{Scaling Elements}
An element $H\in \a$ is called a \emph{scaling element} if 
\[
\hat{\Delta}(\g_0)=\{\alpha\in \hat{\Delta}: \alpha(H)=0\}.
\]
It follows that any scaling element $H\in \a$ centralizes the rootspaces of $\g_0,$ and must be contained in $\z(\g_0).$ There is always at least one scaling element, since the grading element $E\in \a$ is scaling. From the definition, it follows that for $\lambda\in \a^*,$ $\lambda^{\sharp}\in \a$ is scaling if and only if 
\[
\hat{\Delta}(\g_0)=\{\alpha\in \Delta: \langle \alpha,\lambda \rangle=0\}.
\]

Suppose $(\mathcal{G},\omega)$ is a Cartan geometry modeled on $(G,P)$ over $M.$ Suppose $\lambda:G_0\to \R_+$ is a homomorphism such that $(\lambda_*)^{\sharp}$ is a scaling element. There is an associated line bundle $\mathcal{L}^{\lambda}:=\mathcal{G}_0\times_{\lambda} \R_+.$ Automorphisms $\phi\in \mathrm{Aut}(\mathcal{G},\omega)$ act on $\mathcal{L}^{\lambda}$, and an automorphism $\phi$ is called $\lambda$-\emph{inessential} if there exists a global section $f:M\to \mathcal{L}^{\lambda}$ such that $\phi\cdot f=f$. On the other hand, $\phi$ is $\lambda$-\emph{essential} if it is not $\lambda$-inessential. Furthermore, an automorphism $\phi$ is \emph{essential} if it is $\lambda$-essential for every $\lambda$. The following is a consequence of Corollary 6.5 and Definition 7.11 in \cite{Erickson1}, as discussed in \cite{Erickson2}. 
\begin{proposition}\label{prop::EssentialityCriterion}
Suppose $\phi\in \mathrm{Aut}(\mathcal{G},\omega).$ If there exists $e\in \mathcal{G}$ such that $\phi(e)=ep$ for $p\in G_0$ and $p\not\in \ker \lambda$ for any $\lambda:G_0\to \R_{+}$ such that $(\lambda_*)^{\sharp}$ is a scaling element, then $\phi$ is essential.
\end{proposition}

\subsection{Curvature Trees and Harmonic Seeds}
Given a $\g$-valued $2$-form $\Omega\in \ker \Delta$, let $K_{\Omega}=\mathrm{Stab}_{G_0}(\Omega)$ and let $\k_{\Omega}\le \g_0$ be its Lie algebra. A form $\Omega\in (\ker \Delta)_+$ is said to have the \emph{Kruglikov-The property} if 
\begin{enumerate}[(1)]
    \item $\mathrm{im}(\Omega)\subset \g_-\oplus\mathfrak{k}_{\Omega}$,
    \item $\mathrm{im}(\Omega\wedge 1)\subset \ker \Omega.$
\end{enumerate}
A form $\Omega$ with the Kruglikov-The property is called a \emph{harmonic seed} if there exists a model geometry $(J_{\Omega},K_{\Omega})$ and an isomorphism of $K_{\Omega}$-representations $\psi: \mathfrak{j}_{\Omega}\to \g_-\oplus \mathfrak{k}_{\Omega}\le \g$ such that $\mathfrak{j}_{\Omega}$ is the Lie algebra of $J_{\Omega},$ $\psi|_{\k_{\Omega}}=1_{\mathfrak{k}_{\Omega}}$ and $J_{\Omega}/K_{\Omega}$ is simply connected. Let $\omega_{J_{\Omega}}$ and $\omega_P$ be the Maurer-Cartan forms of the Lie groups $J_{\Omega}$ and $P$. Then the Cartan geometry $\mathcal{G}_{\Omega}:=J_{\Omega}\times_{K_{\Omega}}P$ modeled on $(G,P)$ over $J_{\Omega}/K_{\Omega}$ with Cartan connection given by
\[
\omega_{\Omega}=\mathrm{Ad}_{p^{-1}}\psi(\omega_{J_{\Omega}})+\omega_P
\] 
is called the \emph{curvature tree} grown from $\Omega.$ For $j\in J_{\Omega},$ let $L_j:\mathcal{G}_{\Omega}\to \mathcal{G}_{\Omega}$ denote left-action by $j.$ This transformation is right $P$-equivariant. Also,
\[
L_j^*\omega_{\Omega}
=\mathrm{Ad}_{p^{-1}}(\psi(L_j^* \omega_{J_{\Omega}}))+\omega_P
=\mathrm{Ad}_{p^{-1}}(\psi(\omega_{J_{\Omega}}))+\omega_P
=\omega_{\Omega},
\]
so $J_{\Omega}\le \mathrm{Aut}(\mathcal{G}_{\Omega},\omega_{\Omega}).$ The following is Theorem 3.4 in \cite{Erickson2}.
\begin{theorem}\label{thm::IsHarmonicSeed}
Let $\mathfrak{b}_-$ be the nilpotent subalgebra of $\g$ generated by negative restricted root spaces. If $\Omega\in (\ker \Delta)_+$ satisfies the Kruglikov-The property and $\mathrm{im}(\Omega)\subset \mathfrak{b}_-$ then $\Omega$ is a harmonic seed.
\end{theorem}

\subsection{Compact Quotients of Curvature Trees}
Suppose $\Omega$ is a harmonic seed of restricted weight $\tau\in \a^*$ for which $\tau^{\sharp}\in \a$ is not a scaling element. Let $(\mathcal{G},\omega_{\Omega})$ be the curvature tree modeled on $(G,P)$ grown from $\Omega$. By Proposition 4.1 of \cite{Erickson2}, there exists $\alpha\in \hat{\Delta}^+(\p^+)$ and $R\in \g_0^{ss}\cap \a$ such that $a_0:=\alpha^{\sharp}+R\in \ker \tau.$ Then \[
a_0\cdot \Omega=\tau(a_0)\Omega=0.
\]
It follows that $\mathrm{exp}(a_0)\in K_{\Omega}\le J_{\Omega}\le \mathrm{Aut}(\mathcal{G},\omega_{\Omega})$. Fix $\lambda:G_0\to \R_+$ such that $\lambda_*^{\sharp}\in \a$ is a scaling element. We have
    \begin{align*}
    \lambda_*(a_0)
    &=\lambda_*(\alpha^{\sharp})+\lambda_*(R)\\
    &=\alpha(\lambda_*^{\sharp})+\langle \lambda_*^{\sharp},R \rangle\\
    &=\alpha(\lambda_*^{\sharp})
    \neq 0
    \end{align*}
    because $\lambda_*^{\sharp}\in \z(\g_0),$ which is Killing-orthogonal to $R\in\g_0^{ss},$ and because $\lambda_*^{\sharp}$ is a scaling element.
    Then
    \[
    \lambda(\mathrm{exp}(a_0))
    =\mathrm{exp}(\lambda_*(a_0))
    \neq 1,
    \]
    so $\mathrm{exp}(a_0)\not\in \ker \lambda$. Acting on the left by $\mathrm{exp}(a_0)\in G_0$ takes $e\mapsto e \cdot\mathrm{exp}(a_0).$ This transformation is essential by Proposition \ref{prop::EssentialityCriterion}. We have shown the following.
    \begin{proposition}
If $\Omega$ is a harmonic seed of weight $\tau$ for which $\tau^{\sharp}$ is not a scaling element, then the curvature tree grown from $\Omega$ admits an essential transformation.
    \end{proposition}
Under some additional algebraic assumptions, Erickson removes a point from the manifold and quotients by dilation-like transformations to get a compact manifold admitting essential transformations. This process is is similar in spirit to the construction of the Hopf manifold $S^1\times S^{n-1}$ by quotienting $\R^n\backslash \{0\}$ by a discrete group of dilations. The following is a slight modification of Theorem 4.2 of \cite{Erickson2} for which Erickson's proof is still valid.
\begin{theorem}\label{thm::EssentialCompact}
Suppose $\Omega$ is a harmonic seed of weight $\tau$ for which $\tau^{\sharp}\in \a$ is not a scaling element and with constants $a_0:=\alpha^{\sharp}
+R\in \ker \tau\cap \ker \nu_0$ for some $R\in \g_0^{ss}\cap \a$ and $\alpha,\nu_0\in \hat{\Delta}^+(\p^+)$, and $c_0\in \ker(\tau)$ such that $\nu(c_0)>0$ for all $\nu\in \hat{\Delta}^+(\p^+).$ Then there is a one parameter family of essential automorphisms on a nonflat, locally homogeneous, regular, normal Cartan geometry modeled on $(G,P)$ on a manifold diffeomorphic to $S^1 \times S^{\dim(\g_-)-1}.$
\end{theorem}
\begin{remark}
In Theorem 3.5, Erickson proves that for $\Omega$ a lowest weight vector of weight $\tau$ (and satisfying a couple additional properties), the curvature tree $J_{\Omega}$ grown from $\Omega$ has a base space $J_{\Omega}/K_{\Omega}$ diffeomorphic to $\R^n$. It is worth commenting that Erickson's proof of Theorem 4.2 does not depend on this diffeomorphism, and so does not require $\tau$ to be a lowest weight, so long as $\Omega$ is a harmonic seed.
\end{remark}
There is a straightforward proof that every Yamaguchi nonrigid, parabolic model geometry modeled on a homogeneous space for a simple group of real rank at least $3$ has a constant $c_0$ satisfying the requirements. This is Theorem \ref{thm::C0Exists}. However, finding an appropriate $a_0$ was much more difficult for us, and in fact there is one infinitesimal model geometry, $(\sl_4(\H),P_{2,6})$, where a constant $a_0$ satisfying the requirements does not exist for any lowest weight $\Omega$. In this case, we were forced to seek a non-lowest weight harmonic seed. The following two propositions, Proposition 4.3 and Proposition 4.4 of \cite{Erickson2}, facilitate the proof of existence of lowest weights admitting a constant $a_0$ in all other cases.
\begin{proposition}\label{prop::InCenter}
Suppose $\tau^{\sharp}\in \z(\g_0)\cap \a$ is not a scaling element. Then there exist restricted roots $\alpha,\nu_0\in \Delta^+(\p+)$ and $R\in \g_0^{ss}\cap \a$ such that $a_0:=\alpha^{\sharp}+R\in \ker \tau\cap \ker \nu_0.$ 
\end{proposition}
\begin{remark}
It is the case that $\tau^{\sharp}\in \z(\g_0)\cap \a$ exactly when $\langle \tau, \beta_k\rangle=0$ for all $\beta_k\in \hat{\Delta}^0\backslash \hat{I},$ since the corresponding $\beta_k^{\sharp}$ span $\g_0^{ss}\cap \a.$
\end{remark}
\begin{proposition}\label{prop::dimAtLeast2}
Suppose $\tau^{\sharp}\not\in \z(\g_0)$ and $\dim(\g_0^{ss}\cap \a)>1.$ Then for each $\alpha\in \hat{\Delta}^+(\p^+)$, there exists $R\in \g_0^{ss}\cap \a$ and $\nu_0\in \hat{\Delta}^+(\p^+)$ such that $a_0:=\alpha^{\sharp}+R\in \ker \tau \cap \ker \nu_0.$
\end{proposition}
\begin{remark}
    The quantity $\dim(\g_0^{ss}\cap \a)$ is equal to $|\hat{\Delta}^0\backslash \hat{I}|,$ the number of uncrossed vertices in the restricted Dynkin diagram.
\end{remark}
The following is a minor variation of a technique suggested in \cite{Erickson2} on page 20.
\begin{lemma}\label{lemma::twoOrthogonal}
    Suppose $\tau^{\sharp}\not\in \z(\g_0)$ and there exist restricted roots $\nu_0,\alpha\in \hat{\Delta}^+(\p^+)$ such that $\g_0^{ss}\cap \a\subset \ker \nu_0$ and $\langle\nu_0,\alpha\rangle=0.$ Then there exists $R\in \g_0^{ss}\cap \a$ such that $a_0:=\alpha^{\sharp}+R\in \ker \tau\cap \ker \nu_0.$
\end{lemma}
\begin{proof}
If 
\[
\g_0^{ss}\cap \a\subset \ker \tau=(\tau^{\sharp})^{\perp}
\]
then
$\tau^{\sharp}\in (\g_0^{ss}\cap \a)^{\perp}=\z(\g_0)$.
This is not the case, so $\g_0^{ss}\cap \a\not\subset \ker \tau.$ The subspace $\ker \tau \le \a$ has codimension one, so 
\[
\a=\g_0^{ss}\cap \a+\ker \tau.
\]
We have $\alpha^{\sharp}\in \a.$ Then there exists $R\in \g_0^{ss}\cap \a$ such that $\alpha^{\sharp}+R\in \ker \tau.$ On the other hand, $\alpha^{\sharp}\in \ker \nu_0$ and $R\in \ker \nu_0,$ so the claim follows.
\end{proof}
\section{Lowest Weights}
\subsection{Lowest Weights in the Harmonic Curvature Module}
This section develops results enabling us to compute lowest weight vectors in the module of harmonic curvature forms. With these results in hand, we can prove Theorem \ref{thm::SecondOrderCondition} and Theorem \ref{thm::C0Exists}, which are useful for constructing harmonic seeds and compact quotients of curvature trees, respectively. Let $\g$ be real semisimple. We keep in mind the identification $H_{\R}^2(\g_-,\g)\cong \ker \Delta\le C^2(\g_-,\g)$. Given $\beta,\gamma\in \hat{\Delta}^+(\p_+)$ and $\zeta\in \hat{\Delta}$, define a subspace $\g_{\beta,\gamma,\zeta}:=(\g_{\beta})_{\flat}\wedge (\g_{\gamma})_{\flat}\otimes \g_{\zeta}\le \bigwedge^2 (\g_-)^* \otimes \g$.

\begin{proposition}\label{prop::lowestWeightUnique}
Let $V\le \bigwedge^2 (\g_-)^* \otimes \g$ be a $\g_0$-irreducible representation. Suppose there is a $\g_0$ lowest weight vector $v\in V$ contained in $\g_{\beta,\gamma,\zeta}.$ Then every $\g_0$ lowest weight vector in $V$ is contained in $\g_{\beta,\gamma,\zeta}$.
\end{proposition}
\begin{proof}
    We claim that every lowest weight vector in $V$ is in the $\z(\g_0)$ module generated by $v,$ which implies the result because $\z(\g_0)$ is contained in the Cartan subalgebra of $\g.$ Let $w\in V$ be another $\g_0$ lowest weight vector. Let $\g_0^{\ge 0}:=\z(\g_0)\oplus \bigoplus_{\alpha \in \hat{\Delta}^+(\g_0)} \g_{\alpha}.$ Since $v$ is a lowest weight element, an inductive argument shows that $V$, the $\g_0$ module generated by $v$, is in fact the $\g_0^{\ge 0}$ module generated by $v$. The lowest restricted weight of an irreducible representation is unique, so $v$ and $w$ both have the same restricted weight. Because $w$ is in the $\g_0^{\ge 0}$ module generated by $v$ but has the same restricted weight, it must be in the $\z(\g_0)$ module generated by $v.$
\end{proof}
\begin{remark}\label{rmk::WeightIsSum}
Vectors in $\g_{\beta,\gamma,\zeta} $ have weight $\beta+\gamma+\zeta.$
\end{remark}
Let $\iota_{\a}:\h^*\to \a^*$ be the restriction map.
\begin{proposition}\label{prop::DirectSumOfRestricted}
Let $\g\curvearrowright V$ be a representation with complexification $\g^{\C}\curvearrowright V^{\C}$. Let $\lambda\in \a^*.$ Then $\bigoplus_{\widetilde{\lambda}\in \iota_{\a}^{-1}(\lambda)} (V^{\C})_{\widetilde{\lambda}}= (V_{\lambda})^{\C}.$
\end{proposition}
\begin{proof}
We immediately have $(V_{\lambda})^{\C}\le (V^{\C})_{\lambda}.$ Since $\bigoplus_{\lambda\in \a^*} (V_\lambda)^{\C}=V^{\C}=\bigoplus_{\lambda\in \a^*} (V^{\C})_{\lambda}$, none of these inclusions can be proper, and so $V_{\lambda}^{\C}=(V^{\C})_{\lambda}.$

On the other hand, if $\widetilde{\lambda}|_{\a}=\lambda$, then  $(V^{\C})_{\widetilde{\lambda}}\le (V^{\C})_{\lambda},$ which implies \[
\bigoplus_{\widetilde{\lambda}\in \iota_{\a}^{-1}(\lambda)} (V^{\C})_{\widetilde{\lambda}}\le (V^{\C})_{\lambda}.
\] 
Indexing over all $\lambda\in \a^*,$ both sides direct sum to to $V^{\C},$ which implies 
\[
\bigoplus_{\widetilde{\lambda}\in \iota_{\a}^{-1}(\lambda)} (V^{\C})_{\widetilde{\lambda}}
= (V^{\C})_{\lambda}
=(V_{\lambda})^{\C}.
\]
\end{proof}

Specializing to the adjoint representation of $\g$ we have $\bigoplus_{\widetilde{\alpha}\in \iota_{\a}^{-1}(\alpha)}(\g^{\C})_{\widetilde{\alpha}}= (\g_{\alpha})^{\C}.$ In particular, if $\widetilde{\alpha}|_{\a}=\alpha$ then $(\g^{\C})_{\widetilde{\alpha}}\le (\g^{\alpha})^{\C}.$

\begin{remark}\label{rmk::identifications}
Using the identifications $\p^+\cong \g_-^*$ and $\p_+^{\C}\cong (\g_-^{\C})^*$ taking $\eta\mapsto \eta_{\flat_{\R}}$ and $\eta\mapsto \eta_{\flat_{\C}}$, the inclusion $\p_+\to \p_+^{\C}$ becomes $\C$-linear extension $\g_-^*\to (\g^{\C}_-)^*.$ Thus the inclusion map $C_2^{\R}(\p_+,\g)\to C_2^{\C}(\p_+^{\C},\g^{\C})$ taking $\eta_1\wedge_{\R} \eta_2\otimes_{\R} \eta_3\mapsto \eta_1\wedge_{\C} \eta_2\otimes_{\C} \eta_3$ is identified with the extension map $C^2_{\R}(\g_-,\g)\to C^2_{\C}(\g_-^{\C},\g^{\C})$ taking $(\eta_1)_{\flat_{\R}}\wedge_{\R} (\eta_2)_{\flat_{\R}}\otimes_{\R} \eta_3\mapsto (\eta_1)_{\flat_{\C}}\wedge_{\C} (\eta_2)_{\flat_{\C}}\otimes_{\C} \eta_3$. 
\end{remark}
\begin{corollary}
Let $\widetilde{\beta},\widetilde{\gamma}\in \Delta(\p_+^{\C})$ and $\widetilde{\zeta}\in \Delta(\g^{\C}),$ and let $\beta,\gamma\in \hat{\Delta}(\p^+)$ and $\zeta\in \hat{\Delta}(\g)$ be their restrictions to $\a.$ Then $(\g^{\C})_{\widetilde{\beta},\widetilde{\gamma},\widetilde{\zeta}}
\le (\g_{\beta,\gamma,\zeta})^{\C}$.
\end{corollary}
\begin{proof}
    From the proposition we know that
    \[
    (\g^{\C})_{\widetilde{\beta}}\wedge (\g^{\C})_{\widetilde{\gamma}}\otimes 
    (\g^{\C})_{\widetilde{\zeta}}
    \le (\g_{\alpha})^{\C}
    \wedge (\g_{\beta})^{\C}
    \otimes (\g_{\zeta})^{\C}
    =(\g_{\alpha}\wedge \g_{\beta} \otimes \g_{\zeta})^{\C}.
    \]
    The result follows from applying the identifications of Remark \ref{rmk::identifications}.
\end{proof}
The following lemma is useful in the proof of Theorem \ref{thm::LowestWeightIsRestriction}, and in the analysis of the $(\sl_4(\H),P_{2,6})$ case in Section \ref{subsection::sl4}.
\begin{lemma}\label{lemma::LowestIsRestriction}
    Suppose $V\le \bigwedge^2 (\g_-)^* \otimes \g$ is a $\g_0$ subrepresentation, and $\widetilde{\Omega}\in V^{\C}$ is a $\g_0^{\C}$ weight vector such that $\widetilde{\Omega}
    \in \g_{\widetilde{\beta},\widetilde{\gamma},\widetilde{\zeta}}$.
    Then the real and imaginary parts of $\widetilde{\Omega}$ in $V$, if nonzero, are $\g_0$ weight vectors in $\g_{\beta,\gamma,\zeta}$, where $\beta,\gamma,\zeta$ are the restrictions of $\widetilde{\beta},\widetilde{\gamma},\widetilde{\zeta}$ to $\a.$ In addition, if $\widetilde{\Omega}$ is a $\g_0^{\C}$ lowest weight vector, then its real and imaginary parts, if nonzero, are $\g_0$ lowest weight vectors.
\end{lemma}
\begin{proof}
    Without loss of generality, we prove the result for the real part. Define $\Omega,\Omega'\in V$ as the unique vectors such that $\tilde{\Omega}=\Omega+i\Omega'.$ Since $\tilde{\Omega}\in \g_{\widetilde{\beta},\widetilde{\gamma},\widetilde{\zeta}}
    \le(\g_{\beta,\gamma,\zeta})^{\C},$ we have $\Omega,\Omega'\in \g_{\beta,\gamma,\zeta}.$ By assumption, $\Omega\neq 0$. Because $\tilde{\Omega}$ is a weight vector for $\g_0^{\C}$ and $\a\le \h_0$, $\widetilde{\Omega}$ scales by real values under the adjoint action of $\a.$ Therefore $\Omega$ is a weight vector for $\g_0.$ 
    
    Now suppose $\widetilde{\Omega}$ is a lowest weight. Let $\alpha$ be some negative $\g_0$ restricted root. By Proposition \ref{prop::DirectSumOfRestricted}, \[
    ((\g_0)_{\alpha})^{\C}
    =\bigoplus_{\widetilde{\alpha}\in \iota_{\a}^{-1}(\alpha)}(\g_0^{\C})_{\widetilde{\alpha}},
    \]
    a direct sum of negative rootspaces of $\g_0^{\C}.$ These rootspaces annihilate $\tilde{\Omega}$ by assumption, so $(\g_0)_{\alpha}$ annihilates $\tilde{\Omega},$ so $(\g_0)_{\alpha}$ annihilates $\Omega.$ It follows that $\Omega$ is a lowest weight vector. 
\end{proof}
As discussed in Section \ref{section::HarmonicCurvature}, harmonic curvature of regular, normal parabolic geometries is valued in the $\g_0$ module $H^2_{\R}(\g_-,\g)$. However, Kostant's Borel-Weil-Bott facilitates computation of the complex analog $H^2_{\C}(\g_-^{\C},\g^{\C}),$ a $\g_0^{\C}$ module. The following theorem allows us to compare the two modules.
\begin{theorem}\label{thm::LowestWeightIsRestriction}
    Suppose $\g$ is noncomplex simple. There is a $\g_0$ lowest weight vector $\Omega\in H_{\R}^2(\g_-,\g)$ contained in $\g_{\beta,\gamma,\zeta}$ exactly when there is a $\g_0^{\C}$ lowest weight vector $\widetilde{\Omega}\in H^2_{\C}(\g_-^{\C},\g^{\C})$ of the form
    \[
    \Omega=(\eta_{\widetilde{\beta}})_{\flat}
    \wedge(\eta_{\widetilde{\gamma}})_{\flat}
    \otimes \eta_{\widetilde{\zeta}}
    \]
    such that $\beta,\gamma,\zeta$ are the restrictions of $\widetilde{\beta}, \widetilde{\gamma},\widetilde{\zeta}$ to $\a.$
\end{theorem}
\begin{proof}
Suppose 
\[
\tilde{\Omega}=(\eta_{\widetilde{\beta}})_{\flat}
    \wedge(\eta_{\widetilde{\gamma}})_{\flat}
    \otimes \eta_{\widetilde{\zeta}}\in H^2_{\C}(\g_-^{\C},\g^{\C})=H^2_{\R}(\g_-,\g)^{\C}
\]
is a $\g_0^{\C}$ lowest weight vector. By Lemma \ref{lemma::LowestIsRestriction}, there is a $\g_0$ lowest weight vector $\Omega\in H^2_{\R}(\g_-,\g)$ in $\g_{\beta,\gamma,\zeta}.$

Now suppose that there is a lowest weight vector $\Omega\in H^2_{\R}(\g_-,\g)$ in $\g_{\beta,\gamma,\zeta}$. Let $V\le H_{\R}^2(\g_-,\g)$ be a $\g_0$ irreducible subrepresentation containing $\Omega$. By Proposition \ref{prop::ComplexVsRealCohomology}(b),
    \[
    V^{\C}\le H_{\R}^2(\g_-,\g)^{\C}= H^2_{\C}(\g_-^{\C},\g^{\C}).
    \]
    Let $\widetilde{\Omega}\in V^{\C}$ be a $\g_0^{\C}$ lowest weight vector. By the Kostant's Borel-Weil-Bott, $\widetilde{\Omega}
    =(\eta_{\widetilde{\beta}})_{\flat}
    \wedge (\eta_{\widetilde{\gamma}})_{\flat}
    \otimes \eta_{\widetilde{\zeta}}$ for some roots $\widetilde{\beta},\widetilde{\gamma},\widetilde{\zeta}$ of $\g_0^{\C}$. By Lemma \ref{lemma::LowestIsRestriction}, there is a $\g_0$ lowest weight vector $\Omega'\in V$ contained in $\g_{\beta,\gamma,\zeta}$, where $\beta,\gamma,\zeta$ are the restrictions of $\widetilde{\beta},\widetilde{\gamma},\widetilde{\zeta}$ to $\a.$ By Proposition \ref{prop::lowestWeightUnique}, $\Omega\in \g_{\beta,\gamma,\zeta}.$
\end{proof}
Recalling Remark \ref{rmk::WeightIsSum} and Kostant's Borel-Weil-Bott, this implies the following.
\begin{corollary}\label{cor::LowestWeightIsRestriction}
    Suppose $\g$ is noncomplex simple. Then the lowest weights of $H^2_{\R}(\g_-,\g)_+$ are equal to $-(w\cdot \mu)|_{\a}$ for $w\in W^{\p}_+(2).$
\end{corollary}

\subsection{Some Consequences}
\begin{proposition}\label{prop::NegativeCoefficient}
Suppose $\g$ is semisimple of real rank at least $3$ and let $\alpha_i,\alpha_j\in \Delta^0.$ Then 
\begin{enumerate}[(a)]
\item $-s_is_j(\mu)|_{\a}\in - \hat{\Delta}^+,$ and
\item when expressed in terms of simple restricted roots, $-((ij)\cdot \mu)|_{\a}$ has some negative coefficient.
\end{enumerate} 
\end{proposition}
\begin{proof}
We have
\[
s_j(\mu)
=\mu-2\frac{\langle\mu,\alpha_j \rangle}{|\alpha_j|^2}\alpha_j
=\mu-\mu^j \alpha_j,
\]
so
\[
s_is_j(\mu)
=s_i(\mu)-\mu^js_i(\alpha_j)
=\mu-\mu^i\alpha_i-\mu^js_i(\alpha_j).
\]
Then for $\beta_k\in \hat{\Delta}^0$ with $\beta_k\neq \alpha_i|_{\a},\alpha_j|_{\a},$ the expression $-s_is_j(\mu)|_{\a}$ must have a negative $\beta_k$ coefficient when expressed in terms of simple restricted roots. But $-s_is_j(\mu)\in \hat{\Delta},$ so $-s_is_j(\mu)\in -\hat{\Delta}^+,$ proving part (a).

For part (b), notice that
\begin{align*}
(ij)\cdot \mu
&=s_is_j(\mu)+s_is_j(\rho)-\rho\\
&=s_is_j(\mu)-\alpha_i-s_i(\alpha_j).\\
\end{align*} 
Therefore $-(ij)\cdot \mu$ has a negative coefficient associated to $\beta_k$ when expressed in terms of simple restricted roots.
\end{proof}
\begin{remark}
Combining two expressions from the above proof,
\begin{equation}\label{eq::RestrictedAffineAction}
((ij)\cdot \mu)|_{\a}
=\mu|_{\a}-(1+\mu^i)\alpha_i|_{\a}-(1+\mu^j)s_i(e^j)|_{\a}.
\end{equation}
\end{remark}

With this understanding of lowest weight vectors in the harmonic curvature module, we can show that they always satisfy the hypotheses of Theorem \ref{thm::IsHarmonicSeed}, which allows us to construct harmonic seeds.
\begin{theorem}\label{thm::SecondOrderCondition}
    Suppose $\g$ is noncomplex simple with real rank at least $3$ and fixed parabolic subalgebra. Let $\Omega\in H_{\R}^2(\g_-,\g)$ be a $\g_0$ lowest weight, so that
    $\Omega\in \g_{\beta,\gamma,\zeta}$ for some restricted roots $\beta,\gamma,\zeta.$ Then $\zeta\in -\hat{\Delta}^+$ and $\zeta \neq -\beta,-\gamma$.
\end{theorem}
\begin{proof}
It follows from Theorem \ref{thm::LowestWeightIsRestriction} that there is a lowest weight vector $\widetilde{\Omega}=(\eta_{\widetilde{\beta}})_{\flat}
\wedge(\eta_{\widetilde{\gamma}})_{\flat} \otimes \eta_{\widetilde{\zeta}}$ such that $\beta,\gamma,\zeta$ are $\widetilde{\beta},\widetilde{\gamma}, \widetilde{\zeta}$ restricted to $\a.$ By Kostant's Borel-Weil-Bott, we can assume without loss of generality that $\widetilde{\beta}=\alpha_i, \widetilde{\gamma}=s_i(\alpha_j),$ and $\widetilde{\zeta}=-s_is_j(\mu)$ for some $(ij)\in W^{\p}(2),$ so $\beta=\alpha_i|_{\a},$ $\gamma=s_i(\alpha_j)|_{\a},$ and $\zeta=-s_is_j(\mu)|_{\a}.$ It follows from Proposition \ref{prop::NegativeCoefficient}(a) that $\gamma\in -\hat{\Delta}^+.$ It remains to show that 
\[
s_is_j(\mu)|_{\a}\neq \alpha_i|_{\a}, s_i(\alpha_j)|_{\a}.
\]
It follows from the proof of Proposition \ref{prop::NegativeCoefficient} that for $\beta_k\neq \alpha_i|_{\a},\alpha_j|_{\a}$, there is a positive coefficient associated to $\beta_k$ in the left hand side. On the other hand, this coefficient is $0$ in the terms on the right hand side.
\end{proof}
The following corollary parallels arguments from \cite{Erickson2}.
\begin{corollary}\label{cor::IsHarmonicSeed}
Suppose $\g$ is is noncomplex simple with real rank at least $3$ and fixed parabolic subalgebra. Let $\Omega\in H^2_{\R}(\g_-,\g)_+$ be a $\g_0$ lowest weight vector. Then $\im(\Omega)\subset \b_-$ and $\Omega$ satisfies the Kruglikov-The property. 
\end{corollary}
\begin{proof}
By Theorem \ref{thm::SecondOrderCondition}, $\Omega\in \g_{\beta,\gamma,\zeta}$ for $\zeta\in -\hat{\Delta}^+$ and $\zeta\neq -\beta,-\gamma.$ The first condition implies $\mathrm{im}(\Omega)\subset \b_-$. The second condition implies $\im(\Omega\wedge 1)\subset \ker \Omega.$ If $\zeta\in -\hat{\Delta}^+(\p^+)$ then $\mathrm{im}(\Omega)\subset \g_-.$ On the other hand if $\zeta\in \hat{\Delta}(\g_0)$ then $\mathrm{im} (\Omega)\subset \g_0\cap \b_-$, and $\g_0\cap \b_-\subset \mathfrak{k}_{\Omega}$ because $\Omega$ is a $\g_0$ lowest weight vector. Therefore $\mathrm{im}(\Omega)\subset \g_-\oplus \mathfrak{k}_{\Omega}$, and $\Omega$ has the Kruglikov-The property.
\end{proof}
Recall that $E\in \a$ is the grading element. 
\begin{theorem}\label{thm::C0Exists}
Suppose $\tau\in \a^*$ is a restricted weight such that $\tau(E)>0$ and $\tau$ has some negative coefficient when expressed in terms of simple restricted roots. Then there exists $c_0\in \ker \tau$ such that $\nu(c_0)>0$ for all $\nu\in \hat{\Delta}^+(\p^+).$
\end{theorem}
\begin{proof} 
Let 
\[
\mathcal{D}=\{a\in \a: \nu(a)>0 \textrm{ for all } \nu\in \hat{\Delta}^+(\p^+)\}.
\]
Then
\[
\overline{\mathcal{D}}=\{a\in \a:\nu(a)\ge 0 \textrm{ for all } \nu\in \hat{\Delta}^+(\p^+).
\]
 We have $E\in \mathcal{D}$ and $\tau(E)>0.$ As an intersection of halfspaces, $\mathcal{D}$ is connected. Therefore it suffices to find $f\in \mathcal{D}$ such that $\tau(f)<0.$ By the condition on negativity of a certain coefficient, there must be some restricted fundamental weight $\hat{\lambda}_k$ for which 
\[
 \tau(\hat{\lambda}_k^{\sharp})
=\langle \tau,\hat{\lambda}_k\rangle
<0.
\]
We have $\hat{\lambda}_k^{\sharp}\in \overline{\mathcal{D}},$ so by continuity of $\tau$ there exists $f\in \mathcal{D}$ for which $\tau(f)<0$.
\end{proof}
Combining Proposition \ref{prop::NegativeCoefficient} with Theorem \ref{thm::C0Exists} gives the following corollary.
\begin{corollary}\label{cor::C0Exists}
Suppose $\g$ is noncomplex of real rank at least $3$, and $\tau$ is a lowest weight of $H^2_{\R}(\g_-,\g)_+.$ Then there exists $c_0\in \ker \tau$ such that $\nu(c_0)>0$ for all $\nu\in \hat{\Delta}^+(\p^+).$
\end{corollary}
\begin{proof}
     By Corollary \ref{cor::LowestWeightIsRestriction}, the weight $\tau=-((ij)\cdot \mu)|_{\a}$ for some $\alpha_i,\alpha_j\in \Delta^0.$ By Proposition \ref{prop::NegativeCoefficient}, $\tau$ has some negative coefficient when expressed in terms of simple restricted roots. Since $\tau$ has positive homogeneity, $\tau(E)>0.$ The conclusion follows from Theorem \ref{thm::C0Exists}.
\end{proof}

\section{Case Analysis}
\subsection{Non-scaling Weights}
This subsection carries out a case analysis of all Yamaguchi nonrigid geometries with compatible real forms, finding lowest weights $\tau$ in the harmonic curvature module whose duals $\tau^{\sharp}\in \a$ are non-scaling elements. The corresponding lowest weight vectors are associated to curvature trees admitting an essential flow.
\begin{remark}\label{rmk::Adjacency}
Recall equation \eqref{eq::RestrictedAffineAction}. For $\beta_k$ a simple restricted root,
\begin{equation}\label{eq::WeylAction}
    \langle ((ij)\cdot \mu)|_{\a},\beta_k \rangle
=\langle \mu|_{\a},\beta_k \rangle
-(1+\mu^i)\langle \alpha_i|_{\a},\beta_k \rangle
-(1+\mu^j)\langle s_i(\alpha_j)|_{\a}, \beta_k\rangle.
\end{equation}
If $\g$ is any noncomplex simple Lie algebra, then $(\mu|_{\a})^k\ge 0$ for all $\beta_k\in \hat{\Delta}^0$. Then if $\beta_k\neq \alpha_i|_{\a},\alpha_j|_{\a}$ and either $\beta_k$ is adjacent to one of these in the restricted Dynkin diagram or $(\mu|_{\a})^k>0,$ then
\[
\langle ((ij)\cdot \mu)|_{\a}, \beta_k\rangle>0.
\]
\end{remark}

In what follows, we will refer to the set $\{\alpha_i|_{\a},\alpha_j|_{\a}\}$ as the image of $(ij)$ under the restriction to $\a,$ and omit $0$ if it appears.

\begin{lemma}\label{lemma::RealNotScaling}
Suppose $\g$ is noncomplex simple with real rank at least $3$ and a parabolic subalgebra $\p\le \g$. If $(\g,\p)$ is Yamaguchi nonrigid, then there exists $w\in W^{\p}_+(2)$ such that $(w\cdot\mu)|_{\a}^{\sharp}$ is not a scaling element.
\end{lemma}

\begin{proof}
Suppose $(w\cdot\mu)|_{\a}^{\sharp}$ is a scaling element. Then
  \[
  \langle \beta_k, (w\cdot \mu)|_{\a} \rangle=
  \beta_k ((w\cdot\mu)|_{\a}^{\sharp})=0
  \]
  for any $\beta_k\in \hat{\Delta}^0\backslash \hat{I}.$ If there exists $(ij)\in W^{\p}_+(2)$ such that $I\subset \{\alpha_i,\alpha_j\}$, then $\hat{I}\subset \{\alpha_i|_{\a},\alpha_j|_{\a}\}$. By connectedness of the restricted Dynkin diagram and the real rank at least $3$ assumption, there must then be some $\beta_k\in \hat{\Delta}^0\backslash\hat{I}$ which is adjacent to $\{\alpha_i|_{\a},\alpha_j|_{\a}\}$ in the restricted Dynkin diagram. By Remark \ref{rmk::Adjacency}, this implies $\langle (w\cdot \mu)|_{\a}, \beta_k \rangle>0$ and so $(w\cdot \mu)|_{\a}^{\sharp}$ is not a scaling element. In particular, if $|I|=1$ then $I\subset \{\alpha_i,\alpha_j\}$ by Proposition \ref{prop::HasseDescription}. These observations handle cases \textbf{A(1,2,3,4,5,6,7,8,11)}, \textbf{B(1,2,3,4,5,7)}, \textbf{C(1,2,3,4,5,7,8)}, \textbf{D(1,2,3,4,6,8)} and all exceptional cases. The remaining cases are \textbf{A(9,10,12,13,14,15,16)}, \textbf{B(6,8)}, \textbf{C(6,9,10)} and \textbf{D(5,7)}. We now subdivide based on the assumption that the Lie algebra $\g$ is split or non-split.
 \begin{enumerate}[(a)]
 \item \textbf{Split cases}:
 
 In split cases, we will omit restrictions from $\h$ to $\h_0,$ identify $\alpha_k$ with $\beta_k,$ $\Delta$ with $\hat{\Delta},$ and $I$ with $\hat{I}.$ In case \textbf{A(9)}, $\alpha_3$ is adjacent to $(21)\in W^{\p}_+(2).$ By Remark \ref{rmk::Adjacency}, this is sufficient to show that $\langle \alpha_3, (21)\cdot \mu \rangle>0.$ Since $\alpha_3$ is not contained in $I=\{\alpha_2,\alpha_i\},$ we get that $((21)\cdot \mu)^{\sharp}$ is not a scaling element. In case \textbf{A(10)}, the root $\alpha_3$ is adjacent to $(21)$ and not contained in $I=\{\alpha_2,\alpha_{l-1}\}.$ In case \textbf{A(12)}, we have $\alpha_l\not \in I=\{\alpha_1,\alpha_2,\alpha_i\}$ and $\mu^l>0$ and $\alpha_l$ is not in $(12).$ By Remark \ref{rmk::Adjacency}, $((12)\cdot \mu)^{\sharp}$ is not a scaling element. In case \textbf{A(14)}, the root $\alpha_2$ is adjacent to $(1l)$ and not contained in $I=\{\alpha_1,\alpha_i,\alpha_l\}.$ In case \textbf{C(6)}, if $l=3$ then $\alpha_1$ is adjacent to $(23)$ and not contained in $I=\{\alpha_2,\alpha_l\}.$ If $l\ge 4,$ then $\alpha_3$ is adjacent to $(21)$ and not contained in $I=\{\alpha_2,\alpha_l\}.$ In case \textbf{D(5)}, the root $\alpha_3$ is adjacent to $(12)$ and not contained in $I=\{\alpha_1,\alpha_l\}.$ In case \textbf{D(7)}, the root $\alpha_3$ is adjacent to $(12)$ and not contained in $I=\{\alpha_1,\alpha_2,\alpha_l\}.$ The remaining cases are \textbf{A(13,15,16)}, \textbf{B(6,8)} and \textbf{C(9,10)}.
 
In case \textbf{B(6)}, the root $\mu=\lambda_2$ and 
\[
s_3(\alpha_2)=\alpha_2-c_{32}\alpha_3=\alpha_2+2\alpha_3.
\]
Therefore
\begin{align*}
(32)\cdot \mu
&=\mu-(1+\mu^3)\alpha_3-(1+\mu^2)s_3(\alpha_2)\\
&=\lambda_2-\alpha_3-2(\alpha_2+2\alpha_3)\\
&=\lambda_2-2\alpha_2-5\alpha_3.
\end{align*}
Then
\begin{align*}
\langle \alpha_2, (32)\cdot \mu\rangle
&=\frac{|\alpha_2|^2}{2}-2|\alpha_2|^2-\frac{5}{2}c_{23}|\alpha_2|^2\\
&=|\alpha_2|^2\\
&\neq 0
\end{align*}
and $\alpha_2\in \Delta^0\backslash I.$

We will not deal with the other $6$ cases immediately, but will impose some conditions. In case \textbf{A(13)}, if $l\ge 4$ then $\alpha_3$ is adjacent to $(12)$ and not contained in $\{\alpha_1,\alpha_2,\alpha_l\}.$ Therefore $l=3.$ In case \textbf{A(15)} if $i\neq 3$ then $\alpha_3$ is adjacent to $(21)$ and not contained in $I=\{\alpha_1,\alpha_2,\alpha_i,\alpha_j\}.$ Otherwise, $i=3.$ If $j\neq l$, then $\alpha_l$ is not contained in $I=\{\alpha_1,\alpha_2,\alpha_3,\alpha_j\}$ or $(21)$ and $\mu^l>0$. It remains to consider $j=l,$ the case of the form of the form $A_l/P_{1,2,3,l}$ for $l\ge 5.$ In case \textbf{A(16)} we must have $l=4,$ or else we could pick $\alpha_3$ adjacent to $(21)$ and not contained in $I=\{\alpha_1,\alpha_2,\alpha_{l-1},\alpha_l\}.$ In case \textbf{C(9)}, if $l\ge 4$ we can pick $\alpha_3$ adjacent to $(21)$ and not contained in $I=\{\alpha_1,\alpha_2,\alpha_l\}.$ Therefore $l=3.$ In case \textbf{C(10)}, if $i\neq 3$ then $\alpha_3$ is adjacent to $(21)$ and not contained in $I=\{\alpha_1,\alpha_2,\alpha_i\}$. Otherwise, $i=3$. Thus we are considering the geometries $A_3/P_{1,2,3}$ (\textbf{A(13)}), $A_l/P_{1,2,3,l}$ for $l\ge 5$ (\textbf{A(15)}), $A_4/P_{1,2,3,4}$ (\textbf{A(16)}), $B_3/P_{1,2,3}$ (\textbf{B(8)}), $C_3/P_{1,2,3}$ (\textbf{C(9)}), and $C_l/P_{1,2,3}$ for $l\ge  4$ (\textbf{C(10)}). The cases $A_l/P_{1,2,3,l}$ ($l\ge 5$) and $A_4/P_{1,2,3,4}$ both have $(12)\in W^{\p}_+(2),$ so we can consolidate these two cases as $A_l/P_{1,2,3,l}$ for $l\ge 4$ (\textbf{A(15,16)}) and prove that $((12)\cdot \mu)^{\sharp}$ is not a scaling element. We can also consolidate the $C$ cases as $C_l/P_{1,2,3}$ for $l\ge 3$ (\textbf{C(9,10)}) and prove that $((21)\cdot \mu)^{\sharp}$ is not a scaling element.

If $(w\cdot\mu)^{\sharp}$ is a scaling element, we must have 
\[
\langle w\cdot \mu, \alpha\rangle
=\alpha((w\cdot\mu)^{\sharp})\neq 0
\]
for every $\alpha\in \Delta^+(\p^+).$ In particular, we must have $\langle w\cdot \mu, \alpha_k\rangle\neq 0$ for every $\alpha_k\in I.$ In the case $A_3/P_{1,2,3}$ (\textbf{A(13)}), pick $(21)\in W^{\p}_+(2).$ Then
\begin{align*}
(21)\cdot \mu
&=\mu-(1+\mu^2)\alpha_2-(1+\mu^1)s_2(\alpha_1)\\
&=(\alpha_1+\alpha_2+\alpha_3)-\alpha_2-2(\alpha_1+\alpha_2)\\
&=-\alpha_1-2\alpha_2+\alpha_3
\end{align*}
and
\[
\langle \alpha_1,-\alpha_1-2\alpha_2+\alpha_3\rangle=0.
\]
Since $\alpha_1\in I,$ the value $((21)\cdot\mu)^{\sharp}$ is not a scaling element. For $A_l/P_{1,2,3,l}$ (\textbf{A(15,16)}),
\begin{align*}
(21)\cdot \mu
&=\mu-(1+\mu^2)\alpha_2-(1+\mu^1)(s_2(\alpha_1))\\
&=\lambda_1+\lambda_l-\alpha_2-2(\alpha_1+\alpha_2)\\
&=\lambda_1+\lambda_l-2\alpha_1-3\alpha_2
\end{align*}
and
\[
\langle \alpha_1, \lambda_1+\lambda_l-2\alpha_1-3\alpha_2\rangle\\
=\frac{|\alpha_1|^2}{2}(1-4-3c_{12})\\
=0.
\]
Since $\alpha_1\in I,$ the value $((21)\cdot \mu)^{\sharp}$ is not a scaling element. For $B_3/P_{1,2,3}$ (\textbf{B(8)}),
\begin{align*}
(32)\cdot \mu
&=\mu-(1+\mu^3)\alpha_3-(1+\mu^2)s_3(\alpha_2)\\
&=(\alpha_1+2\alpha_2+2\alpha_3)-\alpha_3-2(\alpha_2+2\alpha_3)\\
&=\alpha_1-3\alpha_3
\end{align*}
and
\[
\langle \mu, (32)\cdot \mu \rangle\\
=\langle \lambda_2, \alpha_1-3\alpha_3 \rangle\\
=0.
\] 
Since $\mu\in \Delta^+(\p^+)$, we have shown $(32\cdot \mu)^{\sharp}$ is not a scaling element.
For $C_l/P_{1,2,3}$ (\textbf{C(9,10)}),
\begin{align*}
(21)\cdot \mu
&=\mu-(1+\mu^2)\alpha_2-(1+\mu^1)s_2(\alpha_1)\\
&=2\lambda_1-\alpha_2-3(\alpha_1+\alpha_2)\\
&=2\lambda_1-3\alpha_1-4\alpha_2
\end{align*}
and
\[
\langle \alpha_1,2\lambda_1-3\alpha_1-4\alpha_2\rangle\\
=\frac{|\alpha_1|^2}{2}(2-6-4c_{12})\\
=0.
\]
Since $\alpha_1\in I,$ we have shown $((21)\cdot\mu)^{\sharp}$ is not a scaling element.
 \item \textbf{Non-split cases}:

Again, we only have to consider the cases \textbf{A(9,10,12,13,14,15,16)}, \textbf{B(6,8)}, \textbf{C(6,9,10)} and \textbf{D(5,7)}. Similar to the split version, to show that $(w\cdot \mu)|_{\a}^{\sharp}$ is not a scaling element, it suffices to find $\beta_k\in \hat{\Delta}^0\backslash \hat{I}$ that is either adjacent to (and not equal to) $\{\alpha_i|_{\a},\alpha_j|_{\a}\}$, or for which $(\mu|_{\a})^k>0.$

If $\hat{I}\subset \{\alpha_i|_{\a},\alpha_j|_{\a}\}$ for some $(ij)\in W^{\p}_+(2)$, then there must be a restricted root $\beta_k$ adjacent to $\{\alpha_i|_{\a},\alpha_j|_{\a}\}$ and not contained in $\hat{I}$. In particular, this happens whenever $|I|=1,$ because then $I\subset \{\alpha_i,\alpha_j\}$ by Proposition \ref{prop::HasseDescription}. For a given real form, $I$ must be disjoint from the compact roots in the Satake diagram and if $\alpha\in I$ then $\overline{\alpha}\in I$. Analyzing Satake diagrams of non-split real forms (Appendix A) shows \textbf{A(12,15)}, \textbf{B(6,8)}, \textbf{C(9,10)} are not compatible with such real forms, leaving the cases \textbf{A(9,10,13,14,16)}, \textbf{C(6)} and \textbf{D(5,7)}.

Case \textbf{A(9)} can occur with $\sl(m,\H)$. Case \textbf{A(10)} can occur with $\sl(m,\H),$ or $\su(p,l+1-p)$ for $p\le l/2,$ or $\su(p,p).$ We consider these last two cases simultaneously as $\su(p,l+1-p)$ for $p\le \frac{l+1}{2}$. Case \textbf{A(13)} can only occur with $\su(2,2)$, which has real rank $2,$ less than $3$. Case \textbf{A(14)} can occur with $\su(p,p)$ for $p>2$. Case \textbf{A(16)} can occur with $\su(p,l+1-p)$ for $p\le l/2,$ or $\su(p,p)$. We consider these cases simultaneously as $\su(p,l+1-p)$ for $p\le \frac{l+1}{2}.$ Case \textbf{C(6)} can occur with $\sp(p,p)$, and cases $\textbf{D(5,7)}$ can occur with $\so(3,5)$. 

\begin{center}
\includegraphics[scale=.5]{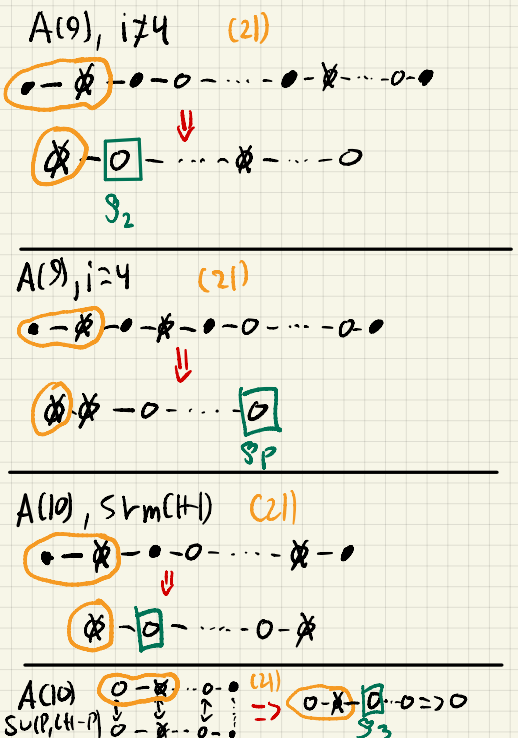}
\end{center}

In case \textbf{A(9)}, the subset $I=\{\alpha_2,\alpha_i\}$. With the real form $\sl(m,\H)$, if $i\neq 4$, then $\hat{I}=\{\beta_1,\beta_j\}$ for $j\neq 2.$ Then $\beta_2$ is adjacent to $\{\beta_1\},$ the image of $(21)$ under restriction, and is not contained in $\hat{I}$. If $i=4,$ then $(\mu|_{\a})^p>0$ and $\beta_p\not\in\hat{I}$ and $\beta_p$ is not in the image of $(21)$. For \textbf{A(10)} with $\sl(m,\H)$ we have $I=\{\alpha_2,\alpha_{l-1}\}$ and $\hat{I}=\{\beta_1,\beta_p\}.$ The image of $(21)$ under restriction is $\{\beta_1\}.$ Since the restricted diagram has at least $3$ simple restricted roots, $\beta_2$ is not contained in $\hat{I}$ and is adjacent to $\beta_1$. For \textbf{A(10)} with $\su(p,l+1-p),$ we have $I=\{\alpha_2,\alpha_{l-1}\}$ and $\hat{I}=\{\beta_2\}.$ Because real rank is at least $3$, the element $\beta_3$ is adjacent to $\{\beta_1,\beta_2\},$ the image of $(21)$ under restriction, and not contained in $\hat{I}.$ For \textbf{A(14)}, we have $I=\{\alpha_1,\alpha_i,\alpha_l\}.$ When considered with the real form $\su(p,p)$ for $p>2,$ we must have $i=p$. Therefore $I=\{\alpha_1,\alpha_p,\alpha_l\}$ and $\hat{I}=\{\beta_1,\beta_p\}.$ Then $\beta_2$ is adjacent to $\{\beta_1\}$, the image of $(1l),$ and not contained in $\hat{I}$. In case \textbf{A(16)} with $\su(p,l+1-p)$ and $p\le \frac{l+1}{2},$ we have $I=\{\alpha_1,\alpha_2,\alpha_{l-1},\alpha_l\}$ and $\hat{I}=\{\beta_1,\beta_2\}.$ Then $\beta_3$ is adjacent to the image of $(21)$ and not contained in $\hat{I}.$ In case \textbf{C(6)} with $\sp(p,p),$ we have $I=\{\alpha_2,\alpha_l\}$ and $\hat{I}=\{\beta_1,\beta_p\}.$ Then $\beta_2$ is adjacent to $\{\beta_1\}$, the image of $(21),$ and is not contained in $\hat{I}.$ In cases $\textbf{D(5,7)}$ with $\so(3,5),$ we rewrite the parabolics using the Dynkin diagram automorphism of $D_4$ switching $\alpha_1$ and $\alpha_3$. Then we have $P_{3,4}$ for $\textbf{D(5)}$ and $P_{2,3,4}$ for \textbf{D(7)}, so that $\alpha_3$ and $\alpha_4$ are related by the bidirectional arrow of the Satake diagram. Then $\hat{I}=\{\beta_3\}$ for \textbf{D(5)} and $\hat{I}=\{\beta_2,\beta_3\}$ for \textbf{D(7)}. Either way, pick $(32)\in W^{\p}_+(2).$ Then $\beta_1$ is adjacent to $\{\beta_2,\beta_3\}$, the image of $(32)$ under restriction, and not contained in $\hat{I}.$

\begin{center}
\includegraphics[scale=.5]{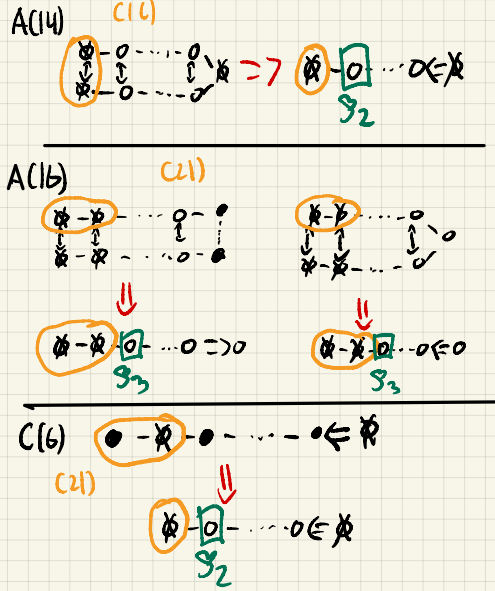}
\end{center}
\end{enumerate}

\end{proof}

\subsection{Existence of $a_0$}
This section shows that the lowest weights from the prior section can be chosen so that an additional condition is satisfied. This condition permits the construction of compact quotients of the associated curvature trees which also admit an essential flow.
\begin{lemma}\label{lemma::a0Exists}
Suppose $\g$ is noncomplex simple of real rank at least $3$ with $\p\le \g$ parabolic. If $(\g,\p)$ is Yamaguchi nonrigid and not isomorphic to $(\sl_4(\H),P_{2,6})$, then there exists $w\in W^{\p}_+(2)$ such that $(w\cdot\mu)|_{\a}^{\sharp}$ is not a scaling element and a constant $a_0:=\alpha^{\sharp}+R\in \ker (w\cdot \mu)|_{\a}\cap \ker \nu_0$ for some $\alpha,\nu_0\in \hat{\Delta}^+(\p^+)$ and $R\in \g_0^{ss}\cap \a.$
\end{lemma}

\begin{proof}
    Lemma \ref{lemma::RealNotScaling} shows that there exists $w\in W^{\p}_+(2)$ such that $(w\cdot\mu)|_{\a}^{\sharp}$ is not a scaling element. If $(w\cdot\mu)|_{\a}^{\sharp}\in \z(\g_0)\cap \a,$ there exists a constant $a_0$ satisfying the required properties by Proposition \ref{prop::InCenter}. Therefore we may assume $(w\cdot\mu)|_{\a}^{\sharp}\not\in \z(\g_0).$

    If $\dim(\g_0\cap \a)>1$ then Proposition \ref{prop::dimAtLeast2} exhibits a constant $a_0$ satisfying the required properties. This condition is equivalent to assuming that that at least two vertices in the restricted Dynkin diagram are uncrossed. Going forward, we may assume there is at most one uncrossed vertex. In particular, since the restricted Dynkin diagram has at least $3$ vertices, this is violated if only one vertex in the Satake diagram is crossed. This handles the cases \textbf{A(1,2,3)}, \textbf{B(1,2,3,4)}, \textbf{C(1,2,3,4)} and \textbf{D(1,2,3,4)}, and all but one exceptional case associated to $G_2$. This $G_2$ case is ruled out by the requirement of real rank at least $3.$ This leaves cases \textbf{A(4,5,6,7,8,9,10,11,12,13,14,15,16)}, \textbf{B(5,6,7,8)}, \textbf{C(5,6,7,8,9,10)}, and \textbf{D(5,6,7,8).}

    If we can find a restricted root $\nu_0\in \hat{\Delta}^+(\p^+)$ vanishing on $\g_0^{ss}\cap \a$ and a restricted root $\alpha\in \hat{\Delta}^+(\p^+)$ such that $\langle \nu_0,\alpha\rangle=0$, then Lemma \ref{lemma::twoOrthogonal} exhibits a constant $a_0$ satisfying the required properties. Vanishing of $\nu_0$ on $\g_0^{ss}\cap \a$ is equivalent to the statement that $\nu_0$ is orthogonal to $\beta_k$ for $\beta_k\in \hat{\Delta}^0\backslash \hat{I},$ since the duals of such $\beta_k$ generate $\g_0^{ss}\cap \a.$

    \begin{enumerate}[(a)]
    \item \textbf{Split cases}:
    
    Suppose first that $\g$ is split. We can rule out any case where $l\ge 4$ and two simple roots are crossed in the Satake diagram, because then $\dim (\g_0^{ss}\cap \a)>1.$ This handles cases \textbf{A(5,6,8,9,10,11)}, \textbf{C(7)}, and \textbf{D(5,6,8)}. In case \textbf{A(14)}, the length $l\ge 5$ and so $\dim(\g_0^{ss}\cap \a)>1.$ In cases \textbf{A(12,13,15,16)}, \textbf{B(8)}, \textbf{C(9,10)}, and \textbf{D(7)}, we may choose $\nu_0=\alpha_1,$ and choose $\alpha$ to be the final crossed root. In case $\textbf{B(5)},$ choose $\nu_0=\alpha_1$ and $\alpha=\mu=\lambda_2$. In case \textbf{B(7)}, choose $\nu_0=\mu=\lambda_2$, and choose $\alpha=\alpha_3$. In case \textbf{C(5)}, choose $\nu_0=\mu=2\lambda_1,$ and choose $\alpha=\alpha_l.$ In case \textbf{C(6)}, we must have $l=3$ by dimensional considerations. Then choose $\nu_0=\alpha_3$ and $\alpha=\mu=2\lambda_1$. In case \textbf{C(7)}, choose $\nu_0=\alpha_l$ and $\alpha=\mu=2\lambda_1$. In case \textbf{C(8)}, choose $\nu_0=\mu=2\lambda_1$ and $\alpha=\alpha_2.$ This leaves cases \textbf{A(4,7)} and \textbf{B(6)}. 

    In case \textbf{B(6)}, let $\nu_0=\alpha_2+2\alpha_3$ and let $\alpha=\alpha_1+\alpha_2+\alpha_3.$ We have $\Delta^0\backslash I=\{\alpha_2\}$ and
    \[
    \langle \alpha_2,\nu_0 \rangle 
    =\frac{|c_2|^2}{2} (c_{22}+2c_{23})=0.
    \]
    Using $2|\alpha_3|^2=|\alpha_2|^2,$
    \begin{align*}
    \langle \nu_0,\alpha \rangle
    &=\langle \alpha_2,\alpha_1\rangle
    +|\alpha_2|^2
    +\langle \alpha_2,\alpha_3 \rangle
    +2\langle \alpha_3,\alpha_2 \rangle
    +2|\alpha_3|^2\\
    &=\frac{|\alpha_2|^2}{2}(c_{21}+c_{22}+3c_{23})+2|\alpha_3|^2\\
    &=\frac{|\alpha_2|^2}{2}(-1+2-3)+|\alpha_2|^2\\
    &=0.
    \end{align*}

    In the remaining cases, $\textbf{A(4,7)},$ we must compute explicit lowest weights. In both cases, $l=3$, or else there is more than one uncrossed simple root. In case $\textbf{A(4)},$ consider the element $(23)\in W^{\p}_+(2).$ This is different from the element we considered for this case in the proof of Lemma \ref{lemma::RealNotScaling}, so we will have to show explicitly that it is not a scaling element. We have
    \begin{align*}
    (23)\cdot \mu
    &=\mu-(1+\mu^2)\alpha_2-(1+\mu^3)s_3(\alpha_2)\\
    &=\mu-\alpha_2-2(\alpha_2+\alpha_3)\\
    &=(\alpha_1+\alpha_2+\alpha_3)-3\alpha_2-2\alpha_3\\
    &=\alpha_1-2\alpha_2-\alpha_3.
    \end{align*}
    Using the Cartan matrix to change to a basis of fundamental weights,
    \begin{align*}
    (23)\cdot \mu
    =(2\lambda_1-\lambda_2)
    -2(-\lambda_1+2\lambda_2-\lambda_3)
    -(-\lambda_2+2\lambda_3)
    =4\lambda_1-4\lambda_2.
    \end{align*}
    Then $\alpha_1+\alpha_2\in \Delta^+(\p^+)$ and using $|\alpha_1|=|\alpha_2|,$
    \begin{align*}
   (\alpha_1+\alpha_2)((23)\cdot \mu)^{\sharp}) 
    &=\langle \alpha_1+\alpha_2, (23)\cdot \mu \rangle\\
    &=4\frac{|\alpha_1|^2}{2}-4\frac{|\alpha_2|^2}{2}\\
    &=0,
    \end{align*}
    so $((23)\cdot \mu)^{\sharp}$ is not a scaling element. On the other hand, since $\Delta^0\backslash I=\{\alpha_3\}$ and
    \[
    \alpha_3(((23)\cdot \mu)^{\sharp})=\langle \alpha_3,(23)\cdot \mu \rangle=0,
    \]
we have $((23)\cdot \mu)^{\sharp}\in \z(\g_0)\cap \a.$ There must exist an appropriate choice of $a_0$ by Proposition \ref{prop::InCenter}.

In case \textbf{A(7)}, pick $(12)\in W^{\p}_+(2).$ Then
\begin{align*}
(12)\cdot \mu
&=\mu-(1+\mu^1)\alpha_1-(1+\mu^2)s_1(\alpha_2)\\
&=(\alpha_1+\alpha_2+\alpha_3)-2\alpha_1-(\alpha_1+\alpha_2)\\
&=-2\alpha_1+\alpha_3.
\end{align*}
Pick $\alpha=\alpha_1+\alpha_2+\alpha_3.$ The root $\alpha_2\in \Delta^0\backslash I.$ Pick $R\in \g_0^{ss}\cap \a$ so that $R_{\flat}=\alpha_2.$ Then $(a_0)_{\flat}:=\alpha+R_{\flat}=\alpha_1+2\alpha_2+\alpha_3.$ Pick $\nu_0=\alpha_1.$ Then
\begin{align*}
\langle (a_0)_{\flat}, \nu_0 \rangle
&=\langle \alpha_1, \alpha_1\rangle
+2\langle \alpha_1,\alpha_2 \rangle\\
&=\frac{|\alpha_1|^2}{2}(c_{11}+2c_{12})\\
&=0.
\end{align*}
Using $|\alpha_1|=|\alpha_2|=|\alpha_3|,$
\begin{align*}
\langle (a_0)_{\flat}, (12)\cdot \mu\rangle
&=\langle\alpha_1+2\alpha_2+\alpha_3, -2\alpha_1+\alpha_3 \rangle\\
&=-2|\alpha_1|^2+\frac{|\alpha_2|^2}{2}(-4c_{21}+2c_{23})+|\alpha_3|^2\\
&=\frac{|\alpha_2|^2}{2}(-4+4-2+2)\\
&=0,
\end{align*}
so $a_0\in \ker ((12)\cdot \mu)\cap \ker \nu_0.$

\item \textbf{Non-split cases}:

The cases \textbf{A(4,5,6,8,12,13,15)}, \textbf{B(6,8)}, \textbf{C(5,7,8,9,10)} are not compatible with any non-split real form of real rank at least $3.$ This leaves cases \textbf{A(7,9,10,11,14,16)}, \textbf{B(5,7)}, \textbf{C(6)}, and \textbf{D(5,6,7,8)}. 

In cases \textbf{A(7,11)}, the crossed roots of the Satake diagram are related by the Satake diagram involution and correspond to a single crossed root in the restricted Dynkin diagram. Thus real rank at least $3$ implies there are at least two uncrossed roots in the restricted Dynkin diagram. Case \textbf{D(5)} is only compatible with the real form $\so(l-1,l+1)$ for $l=4$, that is $\so(3,5).$ However, the associated restricted Dynkin diagram has two uncrossed roots. Case $\textbf{A(10)}$ is compatible with all three non-split real forms: $\sl_{p+1}(\H)$, $\su(p,l+1-p)$ for $p\le l/2$ and $\su(p,p).$ In the latter two cases, there is one crossed root in the restricted Dynkin diagram, and at least two uncrossed roots. In the former case, the number of uncrossed simple restricted roots being at most $1$ forces $p=3,$ which corresponds to the real form $\sl_4(\H).$ Since $I=\{\alpha_2,\alpha_6\}$, this case is excluded by hypothesis. The remaining cases are \textbf{A(9,14,16)}, \textbf{B(5,7)}, \textbf{C(6)} and \textbf{D(6,7,8)}.

In case \textbf{A(14)}, the subset $I=\{\alpha_1,\alpha_i,\alpha_l\}$. The real form must be $\su(p,p)$ with $i=p.$ Then $\hat{I}=\{\beta_1,\beta_p\}.$ In fact $p=3,$ or there will be at least two uncrossed roots in the restricted Dynkin diagram. The restricted diagram is of type $C_3$ and $\hat{\Delta}^0\backslash \hat{I}=\{\beta_2\}.$ Pick $\nu_0=\beta_2+\beta_3$ and $\alpha=\mu|_{\a}=2\beta_1+2\beta_2+\beta_3=2\hat{\lambda}_1$. Then
\begin{align*}
\langle \nu_0, \beta_2 \rangle
&=\frac{|\beta_2|^2}{2}(c_{22}+c_{23})\\
&=\frac{|\beta_2|^2}{2}(2-2)\\
&=0
\end{align*}
and
\[
\langle \nu_0,\alpha\rangle=0.
\]
\begin{center}
\includegraphics[scale=.11]{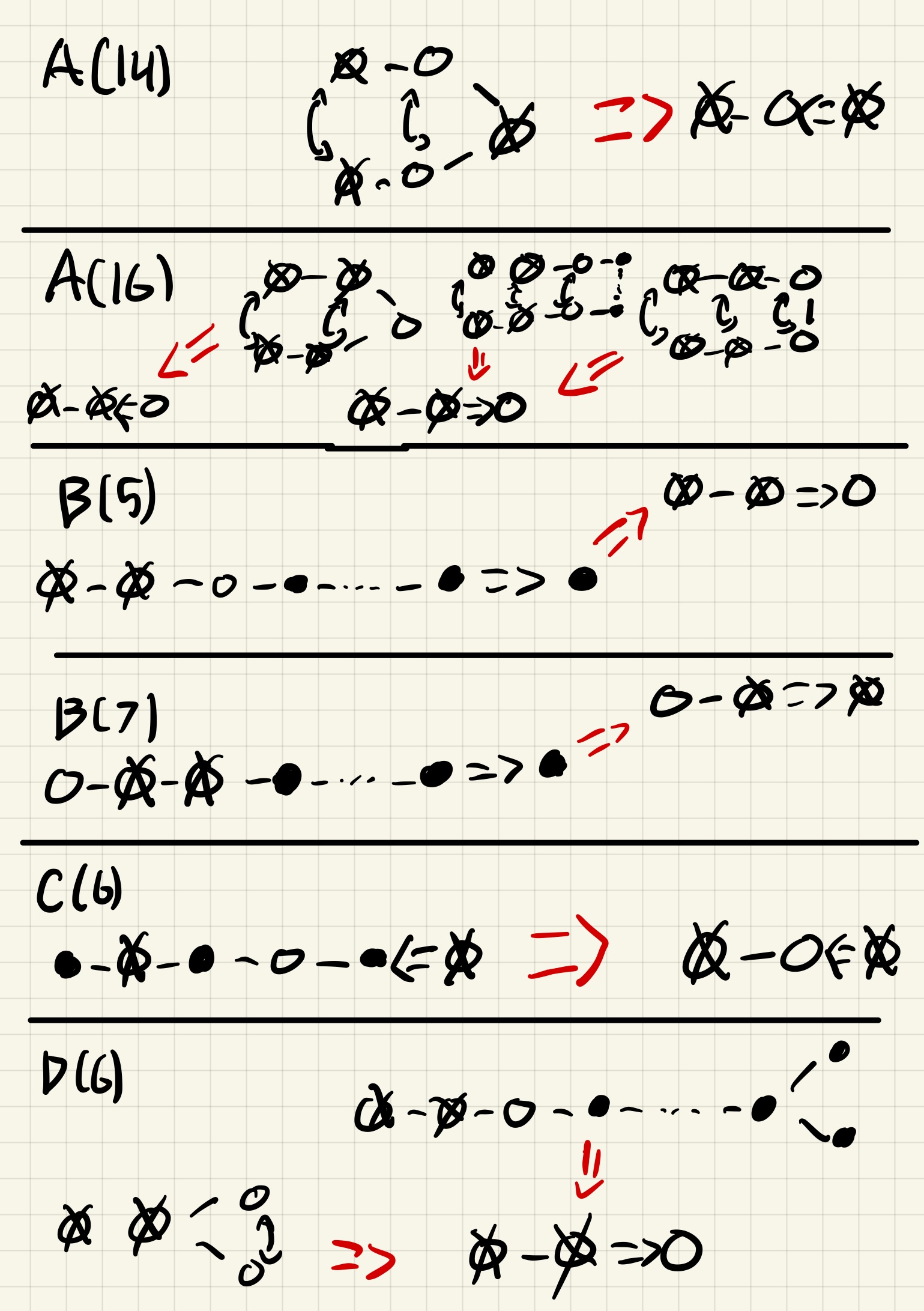}
\end{center}

In case \textbf{A(16)}, we have $I=\{\alpha_1,\alpha_2,\alpha_{l-1},\alpha_l\}.$ The real form must be $\su(p,p)$ or $\su(p,l+1-p)$ for $p\le l/2,$ with $\hat{I}=\{\beta_1,\beta_2\}.$ Since the restricted Dynkin diagram must have at least $3$ nodes with at most $1$ uncrossed, $p=3$ in either case. Then the real form is $\su(3,3)$ or $\su(3,l-2)$ for $l\ge 6.$ In the former case, the restricted diagram has type $C_3.$ In the latter case it has type $B_3.$ Either way, the first two nodes are of the restricted diagram are crossed and $\hat{\Delta}^0\backslash \hat{I}=\{\beta_3\}.$ We can handle both cases simultaneously. The first case has $\mu|_{\a}=2\beta_1+2\beta_2+\beta_3=2\hat{\lambda}_1.$ Similarly, the second case has $\mu|_{\a}=2(\beta_1+\beta_2+\beta_3)=2\hat{\lambda}_1$. Pick $\nu_0=\mu|_{\a}=2\hat{\lambda}_1$ and $\alpha=\beta_2$. Then
\[
\langle \nu_0, \beta_3\rangle=\langle \nu_0,\alpha\rangle=0.
\]
In case \textbf{B(5)}, the subset $I=\{\alpha_1,\alpha_2\}$ and the real form must be $\so(p,2l+1-p)$. Then $\hat{I}=\{\beta_1,\beta_2\}$ and $p=3,$ or at least $2$ nodes in the restricted diagram will be uncrossed. Therefore $\hat{\Delta}^0\backslash \hat{I}=\{\beta_3\}.$ The restricted diagram has type $B_3.$ Pick $\nu_0=\mu|_{\a}=\beta_1+2\beta_2+2\beta_3=\hat{\lambda}_2$ and $\alpha=\beta_1.$ Then
\[
\langle \nu_0, \beta_3\rangle=\langle \nu_0,\alpha\rangle=0.
\]
In case $\textbf{B(7)},$ we have $I=\{\alpha_2,\alpha_3\}$ and the real form must be $\so(p,2l+1-p)$. We have $p=3$, or at least two nodes in the restricted diagram will be uncrossed. The restricted diagram has type $B_3,$ and $\hat{I}=\{\beta_2,\beta_3\}.$ Then $\hat{\Delta}^0\backslash \hat{I}=\{\beta_1\}$. Pick $\nu_0=\mu|_{\a}=\beta_1+2\beta_2+2\beta_3=\hat{\lambda}_2$ and $\alpha=\beta_3.$ Then
\[
\langle \nu_0, \beta_1\rangle=\langle \nu_0, \alpha\rangle=0.
\]
In case \textbf{C(6)}, we have $I=\{\alpha_2,\alpha_l\}$ and the real form must be $\sp(p,p).$ Then $p=3$, or else there will be more than one uncrossed root in the restricted diagram. This real form has a restricted diagram of type $C_3,$ and $\hat{I}=\{\beta_1,\beta_3\}.$ Therefore $\hat{\Delta}^0 \backslash \hat{I}=\{\beta_2 \}.$ Pick $\nu_0=\mu|_{\a}=2\beta_1+2\beta_2+\beta_3=2\hat{\lambda}_1$ and $\alpha=\beta_3.$ Then
\[
\langle \nu_0, \beta_2\rangle=\langle \nu_0, \alpha\rangle=0.
\]
In case \textbf{D(6)}, we have $I=\{\alpha_1,\alpha_2\}$ and the real form must either be $\so(3,2l-3)$ for $l\ge 5$ or $\so(3,5).$ We handle both cases simultaneously. The restricted diagrams have type $B_3,$ and $\hat{\Delta}^0\backslash \hat{I}=\{\beta_3\}.$ Pick $\nu_0=\mu|_{\a}=\beta_1+2\beta_2+2\beta_3=\hat{\lambda}_2$ and $\alpha=\beta_1.$ Then
\[
\langle \nu_0, \beta_3\rangle=\langle \nu_0, \alpha \rangle=0.
\]
In case \textbf{D(7)}, the real form is $\so(3,5),$ but we must use a Dynkin diagram automorphism to rewrite $P_{1,2,4}$ as $P_{2,3,4}$ so that $\alpha_3$ and $\alpha_4$ are related by the Satake diagram's bi-directional arrows. We have $\hat{\Delta}^0\backslash \hat{I}=\{\beta_1\}.$ Pick $\nu_0=\mu|_{\a}=\hat{\lambda}_2$ and $\alpha=\beta_3.$ Then
\[
\langle \nu_0, \beta_1\rangle=\langle \nu_0, \alpha\rangle=0.
\]
\begin{center}
\includegraphics[scale=.5]{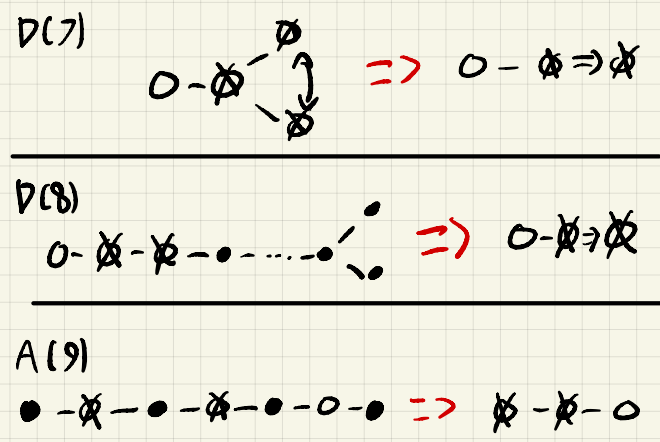}
\end{center}

In case \textbf{D(8)}, the real form must be $\so(p,2l-p)$ for $p<l-1$ or $\so(l-1,l+1)$ for $l\ge 5.$ In the former case, $p=3$, or else there will be more than one uncrossed root. In the latter case, real rank is at least $4$ and so there will be at least $2$ uncrossed roots in the restricted Dynkin diagram. Therefore the real form is $\so(3,2l-3),$ which has a restricted diagram of type $B_3,$ and $\hat{\Delta}^0\backslash \hat{I}=\{\beta_1\}.$ Pick $\nu_0=\mu|_{\a}=\beta_1+2\beta_2+2\beta_3=\hat{\lambda}_2$ and $\alpha=\beta_3.$ Then
\[
\langle \nu_0, \beta_1\rangle=\langle \nu_0, \alpha \rangle=0. 
\]
The only remaining case is \textbf{A(9)}, for which we must compute $(w\cdot \mu)|_{\a}$ for an explicit $w.$ We have $I=\{\alpha_2,\alpha_i\}$ for $i<l-1.$ The real rank must be exactly $3$, or there will be at least $2$ uncrossed roots in the restricted diagram. The real form must be $\sl_4(\H).$ Since $l=7,$ we have $i=4.$ The restricted Dynkin diagram is of type $A_3$ and $\hat{I}=\{\beta_1,\beta_2\},$ so $\hat{\Delta}^0\backslash \hat{I}=\{\beta_3\}.$ For $(21)\in W^{\p}_+(2),$
\begin{align*}
(21)\cdot \mu
&=\mu-(1+\mu^2)\alpha_2-(1+\mu^1)s_2(\alpha_1)\\
&=\mu-\alpha_2-2(\alpha_1+\alpha_2)\\
&=\mu-2\alpha_1-3\alpha_2.
\end{align*}
Then
\[
((21)\cdot \mu)|_{\a}
=(\beta_1+\beta_2+\beta_3)-3\beta_1
=-2\beta_1+\beta_2+\beta_3.
\]
The element $(21)\in W^{\p}_+(2)$ is the only possible choice, and so we already proved this was dual to a non-scaling element in the proof of Lemma \ref{lemma::RealNotScaling}. Let $\alpha=\mu|_{\a}=\beta_1+\beta_2+\beta_3$, and fix $R\in \g_0^{ss}\cap \a$ so that $R_{\flat}=\beta_3.$ Then $(a_0)_{\flat}:=\alpha+R_{\flat}=\beta_1+\beta_2+2\beta_3.$ Let $\nu_0=\beta_1+\beta_2.$ Then using $|\beta_1|=|\beta_2|=|\beta_3|,$
\begin{align*}
\langle (a_0)_{\flat}, ((21)\cdot \mu)|_{\a}\rangle
&=-2|\beta_1|^2+|\beta_2|^2+2|\beta_3|^2
-\langle \beta_1,\beta_2 \rangle
+3\langle \beta_2, \beta_3 \rangle\\
&=\frac{|\beta_2|^2}{2}(2-c_{21}+3{c_{23}})\\
&=0
\end{align*}
and
\begin{align*}
\langle (a_0)_{\flat}, \nu_0\rangle
&=|\beta_1|^2+|\beta_2|^2
+2\langle \beta_1,\beta_2 \rangle
+2\langle \beta_2, \beta_3 \rangle\\
&=\frac{|\beta_2|^2}{2}(4+2c_{21}+2c_{23})\\
&=0,
\end{align*}
so $a_0\in \ker ((21)\cdot \mu)|_{\a}\cap \ker \nu_0.$
\end{enumerate}
\end{proof}

\subsection{The case $(\sl_4(\H),P_{2,6})$}\label{subsection::sl4}
For the case $(\sl_4(\H),P_{2,6})$, we were unable to find an appropriate lowest weight vector and instead settled for a non-lowest weight harmonic seed whose Cartan geometry has a compact quotient admitting essential transformations. We will also write $P_{2,6}$ to represent the corresponding subalgebra of $\sl_8.$ Let $v\in H^2_{\C}((\sl_8)_-,\sl_8)_+$ be a $\g_0$ lowest weight vector associated to $(21)\in W^{\p}_+(2)$ by Kostant's Borel-Weil-Bott. Since
\begin{align*}
(21)(\mu)&
=\mu-\mu^2\alpha_2-\mu_1s_2(\alpha_1)\\
&=\mu-\alpha_1-\alpha_2\\
&=\alpha_3+\alpha_4+\alpha_5+\alpha_6+\alpha_7,
\end{align*}
it is given by
\[
v=(\eta_{\alpha_2})_{\flat}\wedge (\eta_{\alpha_1+\alpha_2})_{\flat}
\otimes \eta_{-\alpha_3-\alpha_4-\alpha_5-\alpha_6-\alpha_7}.
\]
Then
\begin{equation}\label{eqn::ThreeTerms}
\begin{split}
\eta_{\alpha_3+\alpha_4}\cdot v
&=(\eta_{\alpha_2+\alpha_3+\alpha_4})_{\flat}\wedge (\eta_{\alpha_1+\alpha_2})_{\flat}
\otimes \eta_{-\alpha_3-\alpha_4-\alpha_5-\alpha_6-\alpha_7}\\
&+(\eta_{\alpha_2})_{\flat}\wedge (\eta_{\alpha_1+\alpha_2+\alpha_3+\alpha_4})_{\flat}
\otimes \eta_{-\alpha_3-\alpha_4-\alpha_5-\alpha_6-\alpha_7}\\
&+(\eta_{\alpha_2})_{\flat}\wedge (\eta_{\alpha_1+\alpha_2})_{\flat}
\otimes \eta_{-\alpha_5-\alpha_6-\alpha_7}.
\end{split}
\end{equation}
We know
\[
H_{\C}^2((\sl_8)_-,P_{2,6})
\cong H_{\R}(\sl_4(\H)_-,P_{2,6})^{\C},
\]
so that relative to the real form $\bigwedge^2 (\sl_4(\H)_-^*)\otimes \sl_4(\H)\le \bigwedge (\sl_8)_-^*\otimes \sl_8,$ the real and imaginary parts
of $\eta_{\alpha_3+\alpha_4}\cdot v$ are real harmonic elements. Either the real or imaginary part is nonzero, so assume without loss of generality that the real part $\Omega$ is nonzero. Let the real parts of the three terms of equation \eqref{eqn::ThreeTerms} be $\phi_1,\phi_2,\phi_3$, so $\Omega=\phi_1+\phi_2+\phi_3$. Then by Lemma \ref{lemma::LowestIsRestriction}, $\phi_1\in \sl_4(\H)_{\beta_1+\beta_2,\beta_1,-\beta_2-\beta_3}$ and $\phi_2\in \sl_4(\H)_{\beta_1,\beta_1+\beta_2,-\beta_2-\beta_3}$ and $\phi_3\in \sl_4(\H)_{\beta_1,\beta_1,-\beta_3}$ for $\beta_1,\beta_2,\beta_3$ the simple restricted roots of $\sl_4(\H)$. Then $\mathrm{im}(\Omega\wedge 1)\subset \ker \Omega$ and $\mathrm{im}(\Omega)\subset \g_-$, so $\Omega$ has the Kruglikov-The property. Because $\mathrm{im} (\Omega)\le \mathfrak{b_-}$, Theorem \ref{thm::IsHarmonicSeed} implies $\Omega$ is a harmonic seed. The vector $\Omega$ is a restricted weight vector of weight $\tau:=2\beta_1-\beta_3.$ Then $\hat{I}=\{\beta_1,\beta_3\}$ and $\beta_2\in \hat{\Delta}^0\backslash \hat{I}$, so $\beta_2^{\sharp}\in \g_0^{ss}\cap \a$. The restricted diagram has type $A_3.$ From 
\begin{align*}
\langle \beta_2^{\sharp},\tau^{\sharp}\rangle
&=\langle \beta_2,2\beta_1-\beta_3\rangle\\
&=\frac{|\beta_2|^2}{2}(2c_{21}-c_{23})\\
&=\frac{|\beta_2|^2}{2}(-1)\\
&\neq 0,
\end{align*}
it follows that $\tau^{\sharp}
\not\in (\g_0^{ss}\cap \a)^{\perp}
=\mathfrak{z}(\g_0)\cap \a$ and thus $\tau^{\sharp}$ is not a scaling element.

Let $\alpha=\mu|_{\a}=\beta_1+\beta_2+\beta_3\in \hat{\Delta}^+(\p^+)$ and pick $R\in \g_0^{ss}\cap \a$ so that $R_{\flat}=\beta_2.$ Then for $a_0:=\alpha^{\sharp}+R,$
\[
(a_0)_{\flat}=\alpha+R_{\flat}=\beta_1+2\beta_2+\beta_3.
\] 
In terms of restricted fundamental weights,
\[
(a_0)_{\flat}=2\hat{\lambda}_2.
\]
Pick $\nu_0=\beta_1\in \hat{\Delta}^+(\p^+)$.
Then
\[
\langle (a_0)_{\flat},\nu_0 \rangle
=\langle (a_0)_{\flat},\tau\rangle=0,
\]
so $a_0\in \ker \tau \cap \ker \nu_0.$ Observe that $\tau(E)=1>0,$ while the $\beta_3$ coefficient of $\tau$ is negative. By Theorem \ref{thm::C0Exists}, there exists $c_0\in \ker \tau$ such that $\nu(c_0)>0$ for all $\nu\in \hat{\Delta}^+(\p^+).$ Using Theorem \ref{thm::EssentialCompact}, we have proven the following. 

\begin{theorem}\label{thm::sl4H}
Suppose $(G,P)$ is a parabolic model geometry infinitesimally isomorphic to $(\sl_4(\H),P_{2,6}).$ Then there exists a closed, nonflat, locally homogeneous, regular, normal Cartan geometry modeled on $(G,P)$ admitting essential transformations.
\end{theorem}


\section{Main Theorem}
We will start with an analysis of real Lie algebras admitting a complex structure. Let $\s$ be a split real form, and consider the underlying real Lie algebra of $\s^{\C}$. The restricted roots of $\s$ and $\s^{\C}$ are the same, and the restricted rootspaces of $\s$ complexify to the restricted rootspaces of $\s^{\C}.$ By Proposition \ref{prop::ComplexVsRealCohomology}(b) there is a $\s_0$-equivariant injection
\begin{align}\label{eq::RealEmbeds}
H^2_{\R}(\s_-,\s)\hookrightarrow 
H^2_{\R}(\s_-,\s)^{\C}\cong
H_{\C}^2(\s^{\C}_-,\s^{\C}).
\end{align}
For $\Omega\in H^2_{\R}(\s_-,\s),$ let $\Omega'$ be the corresponding element of $H^2_{\C}(\s_-^{\C},\s^{\C}).$ Let $\b'_-$ be the direct sum of the negative rootspaces of $\s^{\C}.$
\begin{proposition}\label{prop::FirstMap}
\begin{enumerate}[(a)]
\item
If $\Omega$ has the Kruglikov-The property, then $\Omega'$ has the Kruglikov-The property.
\item If $\im(\Omega)\subset \b_-$ then $\im(\Omega')\subset \b'_-.$
\end{enumerate}
\end{proposition}
\begin{proof}
\begin{enumerate}[(a)]
    \item The form $\Omega'$ is the $\C$-linear extension of $\Omega,$ so 
\[
\mathrm{im}(\Omega'\wedge 1)
=\mathrm{im} (\Omega \wedge 1)^{\C}
\subset (\ker \Omega)^{\C}
\subset \ker \Omega'.
\]

 Let $\mathfrak{k}_{\Omega'}\le \s_0^{\C}$ be the stabilizer subalgebra of $\Omega'.$ For $X\in \k_{\Omega}$ and $U,V\in \s_-$, we know that 
 \[
 \ad(X) (\Omega')(U,V)=\ad(X)(\Omega)(U,V)=0.
 \]
 Since $\ad(X)(\Omega')$ is $\C$-bilinear, $\ad(X)(\Omega')=0$. Therefore $\k_{\Omega}\le \k_{\Omega'}$, and in fact $\k_{\Omega}^{\C}\le \k_{\Omega'}$. It follows that 
 \[
 \im (\Omega')=\im(\Omega)^{\C} \le (\s_-\oplus \k_{\Omega})^{\C}\le \s_-^{\C}\oplus \k_{\Omega'}.
 \]
    \item The negative rootspaces of $\s^{\C}$ are the complexifications of the negative restricted rootspaces of $\s.$ Therefore 
 \[
 \im (\Omega')=\im(\Omega)^{\C}\subset \b_-^{\C}=\b'_-.
 \]
\end{enumerate}
\end{proof}
By Proposition \ref{prop::RealCohomologyOfComplex} there is an $\s_0^{\C}$-equivariant injection
\begin{align}\label{eq::ComplexReal}
    H^2_{\C}(\s_-^{\C},\s^{\C})\hookrightarrow H^2_{\R}(\s_-^{\C},\s^{\C}).
\end{align}
For $\Omega'\in H^2_{\C}(\s_-^{\C},\s^{\C}),$ let $\Omega''$ be the corresponding element in $H^2_{\R}(\s^{\C}_-,\s^{\C})$. Let $\b''_-$ be the direct sum of the negative restricted rootspaces of $\s^{\C}.$
\begin{proposition}\label{prop::SecondMap}
\begin{enumerate}[(a)]
\item
    If $\Omega'$ has the Kruglikov-The property, then $\Omega''$ has the Kruglikov-The property.
\item If $\im (\Omega')\le \b'_-$ then $\im(\Omega'')\le \b''_-.$
\end{enumerate}
\end{proposition}
\begin{proof}
\begin{enumerate}[(a)]
    \item The map taking $\Omega'$ to $\Omega''$ is the inclusion, so 
    \[  
    \im(\Omega''\wedge 1)
    =\im (\Omega' \wedge 1)
    \subset \ker \Omega'
    =\ker \Omega''.
    \]
    Let $\k_{\Omega''}\le \s_0^{\C}$ be the stabilizer subalgebra of $\Omega''.$ But $\k_{\Omega''}=\k_{\Omega'}$. Then 
    \[
    \im(\Omega'')
    =\im(\Omega')
    \subset \s_-^{\C}\oplus \k_{\Omega'}
    =\s_-^{\C}\oplus \k_{\Omega''}.
    \]

    \item The negative restricted rootspaces of $\s^{\C}$ are equal to the negative rootspaces of $\s^{\C}.$ Therefore
\[
\im (\Omega'')=\im (\Omega')\subset \b'_-=\b''_-.
\]
\end{enumerate}
\end{proof}
     
\begin{lemma}\label{lemma::LowestWeightNotScaling}
Suppose $\g$ is noncomplex simple of real rank at least $3$ and $(\g,\p)\not\cong (\sl_4(\H),P_{2,6})$. Then there exists a lowest weight $\tau$ of $H^2_{\R}(\g_-,\g)_+$ such that $\tau^{\sharp}$ is not a scaling element, a constant $a_0:=\alpha^{\sharp}+R$ such that $a_0\in \ker \tau \cap \ker \nu_0$ for some $\alpha,\nu_0\in \hat{\Delta}^+(\p^+)$ and $R\in \g_0^{ss}\cap \a$, and a constant $c_0\in \ker \tau$ such that $\nu(c_0)>0$ for all $\nu\in \hat{\Delta}^+(\p^+)$.
\end{lemma}
\begin{proof}
Lemma \ref{lemma::a0Exists} shows that there exists $w\in W^{\p}_+(2)$ such that $\tau^{\sharp}:=-(w\cdot\mu)|_{\a}^{\sharp}$ is not a scaling element, and a constant $a_0$ satisfying the required conditions. By Corollary \ref{cor::LowestWeightIsRestriction}, $\tau=-(w\cdot \mu)|_{\a}$ is a lowest weight of $H^2_{\R}(\g_-,\g)_+.$ By Corollary \ref{cor::C0Exists}, there exists a constant $c_0$ satisfying the required conditions.
\end{proof}
    
\begin{lemma}\label{lemma::GoodHarmonicSeeds}
Suppose $\g$ is simple of real rank at least $3$ and $(\g,\p)\not\cong (\sl_4(\H),P_{2,6})$. There exists a harmonic seed $\Omega$ of weight $\tau$ such that $\tau^{\sharp}$ is not a scaling element, a constant $a_0:=\alpha^{\sharp}+R$ such that $a_0\in \ker \tau \cap \ker \nu_0$ for some $\alpha,\nu_0\in \hat{\Delta}^+(\p^+)$ and $R\in \g_0^{ss}\cap \a$, and a constant $c_0\in \ker \tau$ such that $\nu(c_0)>0$ for all $\nu\in \hat{\Delta}^+(\p^+)$.
\end{lemma}
\begin{proof}
If $\g$ is noncomplex, Lemma \ref{lemma::LowestWeightNotScaling} implies the existence of a $\g_0$ lowest weight vector $\Omega\in H^2_{\R}(\g_-,\g)$ of weight $\tau$ such that $\tau^{\sharp}$ is not a scaling element, and constants $a_0$ and $c_0$ satisfying the required conditions. By Corollary \ref{cor::IsHarmonicSeed}, $\im(\Omega)\subset \b_-$ and $\Omega$ satisfies the Kruglikov-The property. By Theorem \ref{thm::IsHarmonicSeed}, $\Omega$ is a harmonic seed.

If $\g$ is the underlying real Lie algebra of a complex Lie algebra, let $\s$ be a split real form such that $\g\cong \s^{\C}.$ Any parabolic subalgebra of $\g$ is induced by a parabolic subalgebra of $\s.$ Because $\s$ is split, it cannot be isomorphic to $\sl_4(\H).$ By Lemma \ref{lemma::LowestWeightNotScaling}, there exists a lowest $\g_0$ weight vector $\Omega$ of $H^2_{\R}(\s_-,\s)$ of weight $\tau$ such that $\tau^{\sharp}$ is not scaling, and constants $a_0:=\alpha^{\sharp}+R$ and $c_0$ satisfying the required conditions. By Corollary \ref{cor::IsHarmonicSeed}, $\Omega$ satisfies the Kruglikov-The property and $\im(\Omega)\subset \b_-.$ Applying the maps in equations \eqref{eq::RealEmbeds} and \eqref{eq::ComplexReal} to $\Omega$ provides an element $\Omega''\in H^2_{\R}(\g_-,\g)$ of weight $\tau.$ By Proposition \ref{prop::FirstMap} and Proposition \ref{prop::SecondMap}, $\Omega''$ satisfies the Kruglikov-The property and $\im(\Omega'')\subset \b''_-.$ Then by Theorem \ref{thm::IsHarmonicSeed}, $\Omega''$ is a harmonic seed. 

Furthermore, the restricted roots of $\g$ are the same as those of $\s,$ and $\a\cap\s_0^{ss}=\a\cap \g_0^{ss}$. The Killing form induced on $\a$ by $\g$ is twice that induced by $\s.$ Therefore $\sharp_{\g}:\a^*\to \a$ is half of $\sharp:\a^*\to \a.$ Therefore $\tau^{\sharp_{\g}}=\frac{1}{2}\tau^{\sharp}$ is not a scaling element and there exist appropriate constants $a''_0=\alpha^{\sharp_{\g}}+\frac{1}{2}R=\frac{1}{2}a_0$ and $c''_0=c_0$.
\end{proof} 

Finally, we combine our analysis of harmonic seeds in the harmonic curvature module with Theorem \ref{thm::EssentialCompact}.
\begin{MainTheorem}[Main Theorem]
    Suppose $(G,P)$ is a Yamaguchi nonrigid parabolic model geometry with $G$ real simple of real rank at least $3$. Then there exists a closed, nonflat, locally homogeneous, regular, normal Cartan geometry modeled on $(G,P)$ admitting essential transformations.
\end{MainTheorem}
\begin{proof}
    If $(G,P)$ is infinitesimally isomorphic to $(\sl_4(\H),P_{2,6})$ the result is established by Theorem \ref{thm::sl4H}, so we may assume this is not the case. By Lemma \ref{lemma::GoodHarmonicSeeds}, there exists a harmonic seed $\Omega\in H^2_{\R}(\g_-,\g)_+$ of weight $\tau$ for which $\tau^{\sharp}$ is not a scaling element, a constant $a_0:=\alpha^{\sharp}+R\in \ker \tau \cap \ker \nu_0$ for some $\alpha,\nu_0\in \hat{\Delta}^+(\p^+)$ and $R\in \g_0^{ss}\cap \a$, and a constant $c_0\in \ker \tau$ such that $\nu(c_0)>0$ for all $\nu\in \hat{\Delta}^+(\p^+).$ The conclusion follows from Theorem \ref{thm::EssentialCompact}.
\end{proof}

\section*{Appendix A: Real Forms and Restricted Dynkin Diagrams}

The following Satake diagrams come from \cite{CapSlovak}.

\includegraphics[scale=.8]{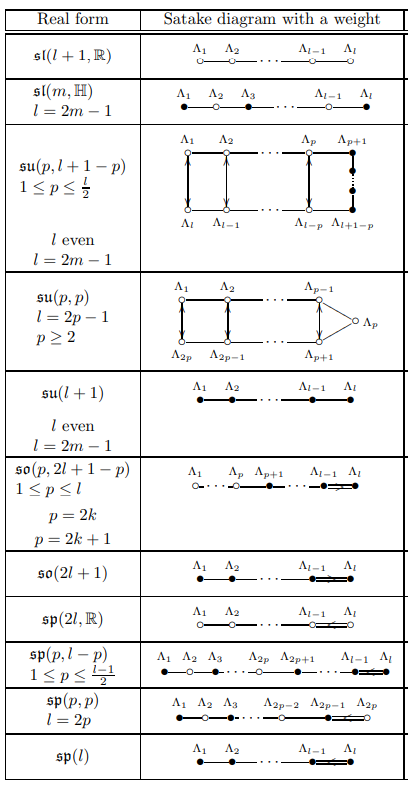}
\includegraphics[scale=.8]{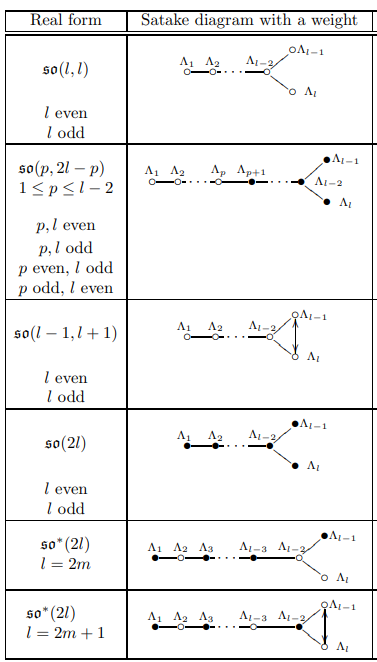}

\includegraphics[scale=.8]{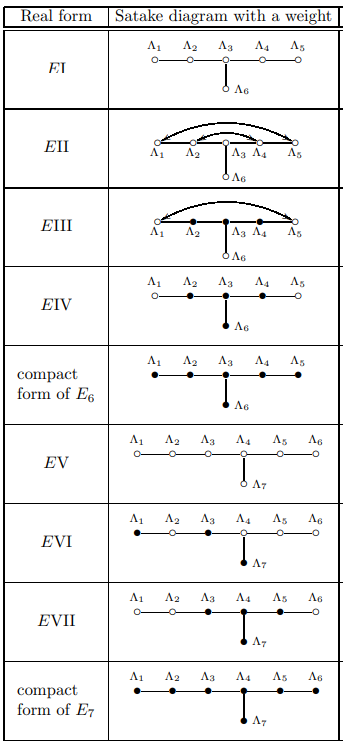}
\includegraphics[scale=.8]{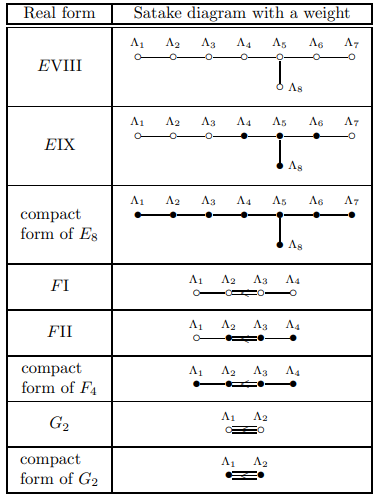}

\textbf{Classical, non-split, noncompact real forms}

The information on the type of the restricted diagrams is from \cite{Satake}.
\begin{center}
\scriptsize
\begin{tabular}{|m{7em} m{4em}  m{4em} m{4em} m{6em}|} 
\hline
  Real Form & Dynkin Diagram & Restricted Diagram & $\mu|_{\a}$& $1\le p$\\
  \hline
    $\sl(p+1,\H)$ & $A_{2p+1}$ & $A_p$&$\hat{\lambda}_1+\hat{\lambda}_p$&\\
    \hline
    $\su(p,l+1-p)$ & $A_l$ & $B_p$& 2$\hat{\lambda}_1$ &$p\le l/2$\\
    \hline
    $\su(p,p)$ & $A_{2p-1}$ & $C_p$& $2\hat{\lambda}_1$&\\
    \hline
    $\so(p,2l+1-p)$& $B_l$ & $B_p$& $\hat{\lambda}_2$ &$3,p \le l$\\
    \hline
    $\sp(p,l-p)$& $C_l$ & $B_p$& $2\hat{\lambda}_1$& $p \le \frac{l-1}{2}$\\
    \hline
    $\sp(p,p)$& $C_{2p}$& $C_p$& $2\hat{\lambda}_1$&\\
    \hline
    $\so(p,2l-p)$& $D_{l}$& $B_p$&$\hat{\lambda}_2$ &$2,p \le l-2$\\
    \hline
    $\so(p,p+2)$& $D_{p+1}$& $B_p$ &$\hat{\lambda}_2$& $3\le p$\\
    \hline
    $\so^*(4p)$& $D_{2p}$&$C_p$& $2\hat{\lambda}_1$& $2\le p$\\
    \hline
    $\so^*(4p+2)$&$D_{2p+1}$& $B_p$& $2\hat{\lambda}_1$& $2\le p$\\
    \hline
\end{tabular}
\\
\end{center}

\normalsize
\section*{Appendix B: Yamaguchi Nonrigid Geometries}
We renumbered the type A geometries so that case \textbf{A(11)}$^\circ$ becomes case \textbf{A(12)} and numbers of later type A cases are increased by one. In case \textbf{A(14)}, we replaced $i\le l/2$ with $i\le \frac{l+1}{2}.$ We made a small correction to case \textbf{A(15)}, so that it doesn't overlap with case \textbf{A(16)}. We follow the numbering convention for Dynkin diagram nodes set in \cite{CapSlovak}. This differs from the numbering used by Yamaguchi in \cite{Yamaguchi1} and \cite{Yamaguchi2} only for the exceptional diagrams.
\\
\scriptsize
\begin{center}
\begin{tabular}{|m{2em} m{4em}  m{3cm} m{2cm}|} 
\hline
  $A_l$ & $P_I$ & $W^{\p}_+(2)$ & $l\ge 2$\\
  \hline
    \textbf{(1)} & $P_1$ & $(12)$ &\\
    \hline
    \textbf{(2)} & $P_2$ & $(21),(23)$ & $l\ge 3$\\
    \hline
    \textbf{(3)}& $P_i$ & $(i\ i-1),(i\ i+1)$ & $2<i \le \frac{l+1}{2}$\\
    \hline
    \textbf{(4)}& $P_{1,2}$ & $(12),(21)$ & $l\neq 3$\\
    & & $(12),(21),(23)$ & $l=3$\\
    \hline
    \textbf{(5)} & $P_{1,i}$ & $(1i)$ & $1<i<l-1$\\
    \hline
   \textbf{(6)}& $P_{1,l-1}$ & $(12),(1\ l-1),(l-1\ l)$ & $l\ge 4$\\
    \hline
    \textbf{(7)}& $P_{1,l}$ & $(12),(l\ l-1),(1l)$ & $l\ge 3$\\
    \hline
    \textbf{(8)}& $P_{2,3}$ & $(21),(23),(32),(34)$ & $l=4$\\
    & & $(21),(23),(32)$ & $l\ge 5$\\
    \hline
    \textbf{(9)}& $P_{2,i}$ & $(21)$ & $3<i<l-1$\\
    \hline
    \textbf{(10)}& $P_{2,l-1}$ & $(21),(l-1 \ l)$ & $l\ge 5$\\
    \hline
    \textbf{(11)}& $P_{i,i+1}$ & $(i\ i+1),(i+1\ i)$ & $2<i \le l/2$\\
    \hline
    \textbf{(12)}& $P_{1,2,i}$ & $(12),(21)$ & $2<i<l$\\
    \hline
    \textbf{(13)}& $P_{1,2,l}$ & $(13),(12),(32),(21),(23)$ & $l=3$\\
    & & $(1l),(12),(21)$ & $l\ge 4$\\
    \hline
    \textbf{(14)}& $P_{1,i,l}$ & $(1l)$ & $2<i\le \frac{l+1}{2}$\\
    \hline
    \textbf{(15)}& $P_{1,2,i,j}$ & $(21)$ & $2<i<j,l-1$\\
    \hline
    \textbf{(16)}& $P_{1,2,l-1,l}$ & $(21),(l-1\ l)$& $l\ge 4$\\ 
    \hline
\end{tabular}
\\

\begin{tabular}{|m{2em} m{4em}  m{3cm} m{2cm}|} 
\hline
  $B_l$ & $P_I$ & $W^{\p}_+(2)$ & $l\ge 3$\\
  \hline
  \textbf{(1)} & $P_1$ & $(12)$&\\
  \hline
  \textbf{(2)} & $P_2$& $(21),(23)$&\\
  \hline
  \textbf{(3)} &$P_3$ &$(32)$ &\\
  \hline
  \textbf{(4)} & $P_l$& $(l \ l-1)$ & $l\ge 4$\\
  \hline
  \textbf{(5)} &$P_{1,2}$ & $(21),(12)$&\\
  \hline
  \textbf{(6)} & $P_{1,3}$& $(32)$ & $l=3$\\
  \hline
  \textbf{(7)} & $P_{2,3}$ & $(32)$ & \\
  \hline
  \textbf{(8)} & $P_{1,2,3}$ & $(32)$& $l=3$\\
  \hline
  \end{tabular}
  \\
  
  \begin{tabular}{|m{2em} m{4em}  m{3cm} m{2cm}|} 
  \hline
 $C_l$ & $P_I$ & $W^{\p}_+(2)$ & $l\ge 2$\\
 \hline
  \textbf{(1)} & $P_l$& $(l \ l-1)$ & \\
  \hline
  \textbf{(2)} &$P_1$ & $(12)$&\\
  \hline
  \textbf{(3)} & $P_2$& $(21),(23)$& $l=3$\\
  & & $(21)$ & $l\ge 4$\\
  \hline
  \textbf{(4)} & $P_{l-1}$ & $(l-1 \ l)$ & $l\ge 4$\\
  \hline
  \textbf{(5)} & $P_{1,l}$& $(12),(21)$& $l=2$\\
  & & $(1l),(12)$& $l\ge 3$\\
  \hline
  \textbf{(6)} & $P_{2,l}$& $(21),(23)$ & $l=3$\\
  & & $(21)$ & $l\ge 4$\\
  \hline
  \textbf{(7)} & $P_{l-1,l}$& $(l-1 \ l)$& $l\ge 4$\\
  \hline
  \textbf{(8)} & $P_{1,2}$& $(12),(21)$ & $l\ge 3$\\
  \hline
  \textbf{(9)} & $P_{1,2,l}$& $(21)$& $l\ge 3$\\
  \hline
  \textbf{(10)} & $P_{1,2,i}$& $(21)$ & $2<i<l$\\
  \hline
\end{tabular}
\\

\begin{tabular}{|m{2em} m{4em}  m{3cm} m{2cm}|}
\hline
  $D_l$ & $P_I$ & $W^{\p}_+(2)$ & $l\ge 4$\\
  \hline
  \textbf{(1)} & $P_1$& $(12)$&\\
  \hline
  \textbf{(2)} & $P_l$ & $(l \ l-2)$& $l\ge 5$\\
  \hline
  \textbf{(3)} & $P_2$& $(21),(23),(24)$ & $l=4$\\
  & & $(21),(23)$ & $l\ge 5$\\
  \hline
  \textbf{(4)} & $P_3$& $(32)$ & $l\ge 5$\\
  \hline
  \textbf{(5)} & $P_{1,l}$& $(12),(42)$& $l=4$\\
  & & $(12)$ & $l\ge 5$\\
  \hline
  \textbf{(6)} & $P_{1,2}$& $(12),(21)$&\\
  \hline
  \textbf{(7)} & $P_{1,2,l}$& $(12),(42)$ & $l=4$\\
  & & $(12)$ & $l\ge 5$\\
  \hline
  \textbf{(8)} & $P_{2,3}$& $(32)$ &$l\ge 5$\\
  \hline
  \end{tabular}
  \\
  
\begin{tabular}{|m{3em} m{4em}  m{1cm}|} 
\hline
Excep. & $P_I$ & $W^{\p}_+(2)$\\
\hline
\textbf{(1)} & $E_6/P_1 $& $(12)$\\
\hline
\textbf{(2)} & $E_6/P_6 $& $(63)$\\
\hline
\textbf{(3)} & $E_7/P_1 $& $(12)$\\
\hline
\textbf{(4)} & $E_7/P_6$& $(65)$\\
\hline
\textbf{(5)} & $E_8/P_7 $&$(76)$ \\
\hline
\textbf{(6)} & $F_4/P_4 $&$(43)$ \\
\hline
\textbf{(7)} & $G_2/P_1 $&$(12)$ \\
\hline
\textbf{(8)} & $G_2/P_2 $&$(21)$ \\
\hline
\textbf{(9)} & $G_2/P_{1,2} $& $(12)$\\
\hline
\end{tabular}
\\
\end{center}


\normalsize



\begin{thebibliography}{9}
\bibitem{Alt}
Alt, J. (2011). \emph{Essential parabolic structures and their infinitesimal automorphisms}. Symmetry, Integrability, and Geometry: Methods and Applications, Vol 7.

\bibitem{CapSlovak}
\v Cap, A., Slov\'ak, J. (2000). \emph{Parabolic Geometries I: Background and General Theory}. Mathematical Surveys and Monographs, Vol 154.

\bibitem{CaseCurryMatveev}
Case, J., Curry, S., Matveev, V. (2018). \emph{On the Lichnerowicz conjecture for CR manifolds with mixed signature.} Comptes Rendus Mathematique, Vol 365, Iss 5, pp 532-537.

\bibitem{DambraGromov}
D'Ambra, G., Gromov, M. (1991). \emph{Lectures on transformation groups: geometry and dynamics}. Surveys in Differential Geometry, Vol 1, pp 19-111.

\bibitem{Erickson1}
Erickson, J. (2022). \emph{Intrinsic holonomy and curved cosets of Cartan geometries}. European Journal of Mathematics, Vol 8, pp 446-474.

\bibitem{Erickson2}
Erickson, J. (2025). \emph{Higher rank parabolic geometries with essential automorphisms and nonvanishing curvature}. Transformation Groups, Vol 30, pp 203-234.

\bibitem{Frances1}
Frances, C. (2007). \emph{Sur le groupe d'automorphismes des g\'eom\'etries paraboliques de rang un}. Annales scientifiques de l'\'Ecole Normale Sup\'erieure, Vol 40. [English: A Ferrand Obata Theorem for Rank One Parabolic Geometries.]

\bibitem{Frances2}
Frances, C. (2015). \emph{About pseudo-Riemannian Lichnerowicz conjecture}. Transformation Groups, Vol 20, pp 1015-1022.

\bibitem{KruglikovThe}
Kruglikov, B., The, D. (2017). \emph{The gap phenomenon in parabolic geometries}. Journal f\"ur die reine und angewandte Mathematik, pp 153-215.

\bibitem{Lelong-Ferrand}
Lelong-Ferrand J. (1971). \emph{Transformations conformes et
quasi-conformes des vari\'et\'es riemanniennes compactes
(d\'emonstration de la conjecture de A. Lichnerowicz)}. Acad\'emie royale de Belgique, Vol 39, pp 5-44.

\bibitem{Mehidi}
Mehidi, L. (2025). \emph{Conformal quotients of plane waves, and Lichnerowicz conjecture in a locally homogeneous setting.} arXiv:2503.08614.

\bibitem{MelnickFrances}
Melnick, K., Frances, C. (2021). \emph{The Lorentzian Lichnerowicz Conjecture for real-analytic three-dimensional manifolds}. arXiv:2108.07215.

\bibitem{MelnickPecastaing}
Melnick, K., Pecastaing, V. (2022). \emph{The conformal group of a compact simply connected conformal manifold.} Journal of the American Mathematical Society, Vol 35, pp 81-122.

\bibitem{Obata}
Obata, M. (1971). \emph{The conjectures on conformal transformations of Riemannian manifolds}. Journal of Differential Geometry, Vol 6, pp 247-258.

\bibitem{Pecastaing1}
Pecastaing, V. (2018). \emph{Lorentzian manifolds with a conformal action of $\mathrm{SL}(2,\R)$}. Commentarii Mathematici Helvetici, Vol 93, pp 401-439.

\bibitem{Pecastaing2}
Pecastaing, V. (2023). \emph{Conformal actions of solvable Lie groups on closed Lorentzian manifolds.} arXiv:2307.05436.

\bibitem{Satake}
Satake, I. (1960). \emph{On representations and compactifications of symmetric Riemannian spaces.} Annals of Mathematics, Vol 71, No. 1, pp 77-110.

\bibitem{Schoen}
Schoen, R. (1995). \emph{On the conformal and CR automorphism groups}. Geometric and Functional Analysis, Vol 5, pp 464-481.

\bibitem{Webster}
Webster, S. (1978). \emph{Pseudo-Hermitian structures on a real hypersurface}. Journal of Differential Geometry, Vol 13, pp 25-41.

\bibitem{Yamaguchi1}
Yamaguchi, K. (1993). \emph{Differential systems associated with simple graded Lie algebras.} Advanced Studies in Pure Mathematics, Progress in Differential Geometry, Vol 22, pp 413-494.

\bibitem{Yamaguchi2}
Yamaguchi, K. (1998). \emph{$G_2$-geometry of overdetermined systems of second order}. Hokkaido University Preprint Series, Vol 417.
\end{thebibliography}
\end{document}